\documentclass[12pt]{amsart}
\usepackage[utf8]{inputenc} 
\usepackage[total={6.5in,9in},top=1in, left=1in, right=1in, bottom=1in]{geometry}
\usepackage{graphicx,amsmath,mathtools,amssymb,epsfig,color}
\usepackage{caption}
\usepackage{subcaption}
\usepackage[abs]{overpic}
\usepackage{adjustbox}
\usepackage[cmyk]{xcolor}

\usepackage[allcolors = blue,colorlinks]{hyperref}

\usepackage{bbold}
\usepackage{dsfont} 
\usepackage{graphicx,float}
\usepackage{graphicx}
\usepackage{epsfig,color,fancyhdr,setspace}
\usepackage{epstopdf}
\usepackage{wrapfig}
\usepackage{soul}
\usepackage{import}
\usepackage{lineno}
\usepackage{stackrel}

\usepackage[abs]{overpic}

\usepackage{wasysym}
\usepackage{marvosym}
\usepackage{extarrows}
\usepackage{amsmath}
\usepackage{graphicx,amsmath,mathtools,float}
\usepackage{setspace}

\usepackage{stackrel}
\usepackage{yfonts}
 
\usepackage{blindtext}
\usepackage{scalerel,stackengine}
\usepackage{tikz}

\newtheorem{theorem}{Theorem}[section]
\newtheorem{lemma}[theorem]{Lemma}

\theoremstyle{definition}
\newtheorem{definition}[theorem]{Definition}
\newtheorem{example}[theorem]{Example}

\theoremstyle{remark}

\numberwithin{equation}{section}

\newtheorem{proposition}[theorem]{Proposition}

\newtheorem{conjecture}[theorem]{Conjecture}

\begin{document}
\title{A Study of Gram Determinants in Knot Theory}

\author{Dionne Ibarra}
\address{School of Mathematics, Monash University, Australia.}
\email{{\rm \textcolor{blue}{dionne.ibarra@monash.edu}}}
	
\author{Gabriel Montoya-Vega}
\address{Department of Mathematics, University of Puerto Rico at R\'io Piedras, San Juan, PR, USA }
\email{{\rm \textcolor{blue}{gabrielmontoyavega@gmail.com $|$ gabriel.montoya@upr.edu}}}

	\keywords{Gram determinants, knot theory, relative Kauffman bracket skein module}
	
	\date{\today}

 \subjclass[2020]{Primary: 57K10. Secondary: 57M27.}

	\maketitle

\begin{abstract}
Historically originated as a sub-field of topology, knot theory is an active area of mathematical investigation that has strong connections with a diverse set of scientific fields such as algebra, biology, and statistical mechanics. A popular and important concept in linear algebra, Gram determinants enjoy a connection with the mathematical theory of knots. In this article, we expose this concept and present several types of Gram determinants in what can be considered as a survey of the current Gram determinants of interest to knot theorists; examples are included to illustrate the definitions. In particular, we pay special attention to a recently defined determinant from a M\"obius band and we further study its structure. At the end, some speculation is presented regarding the closed formula for the Gram determinant of type $(Mb)_1$, a problem that arouses serious interest among knot theorists. 
\end{abstract}

\tableofcontents

\section{Introduction to Gram Determinants in Knot Theory}
The mathematical theory of knots has its historical origins as a sub-area of topology, arguably because of Leibniz's desire that a different type of analysis was needed; he called it \textit{analysis situs} (geometry of position). Although the theory quickly developed into its own research field, it enjoys broad connections to other areas like graph theory, algebra, statistical mechanics, and mathematical biology. In this section we explore how Gram determinants are related to knot theory and we present different types of this determinant. It is important to mention that the initial interest, and most of the current research in Gram determinants related to knot theory, is owed to Edward Witten's belief on the existence of a $3$-manifold invariant related to the Jones polynomial. In particular, it was Nicolai Reshetikhin's and Vladimir Turaev's construction of such invariant that triggered a number of subsequent works by knot theorists on the relation of these determinants with the mathematical theory of knots.

\ 

The notion of a relative skein module is needed in the study of Gram determinants in knot theory. The most popular of such structures is the Kauffman bracket skein module; its structure has shown connections between the module and the geometry and topology of the 3-manifold  (see for example \cite{Prz1}). The relative Kauffman bracket skein module is of our interest and it is introduced in Definition \ref{rkbsm}.  

\begin{definition}\label{rkbsm}
Let $M$ be an oriented $3$-manifold and  $\{x_i\}_{1}^{2n}$ be the set of $2n$ framed points on $\partial M$. Let $I=[-1, 1]$, and let $\mathcal{L}^{\mathit{fr}}(2n)$ be the set of all relative framed links (which consists of all framed links in $M$ and all framed arcs, $I \times I$, where $ I \times \partial I$ is connected to framed points on the boundary of $M$) up to ambient isotopy while keeping the boundary fixed in such a way that $L \cap \partial M = \{x_i\}_{1}^{2n}$. Let $R$ be a commutative ring with unity,  $A \in R$ be invertible, and let $S_{2,\infty}^{\mathit{sub}}(2n)$ be the submodule of $R\mathcal{L}^{\mathit{fr}}(2n)$ that is generated by the Kauffman bracket skein relations:

\begin{itemize}
    \item [(i)]$L_+ - AL_0 - A^{-1}L_{\infty}$, and
    \item [(ii)] $L \sqcup \pmb{\bigcirc}  + (A^2 + A^{-2})L$,
\end{itemize}
where \pmb{$\bigcirc$} denotes the framed unknot and the {\it skein triple} $(L_+$, $L_0$, $L_{\infty})$ denotes three framed links in $M$ that are identical except in a small $3$-ball in $M$ where the difference is shown in Figure \ref{KBSM:skeintriple}. Then, the \textbf{relative Kauffman bracket skein module} (RKBSM) of $M$ is the quotient: $$\mathcal{S}_{2,\infty}(M, \{x_i\}_1^{2n}; R, A) = R\mathcal{L}^{\mathit{fr}}(2n) / S_{2,\infty}^{\mathit{sub}}(2n).$$ 
\end{definition}

\begin{figure}[ht]
\centering
\begin{subfigure}{.25\textwidth}
\centering
$\vcenter{\hbox{\includegraphics[scale = .45]{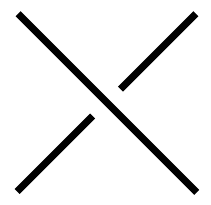}}} $
\caption{$L_+$.} \label{KBSM:Lplus}
\end{subfigure}
\centering
\begin{subfigure}{.25\textwidth}
\centering
$\vcenter{\hbox{\includegraphics[scale = .45]{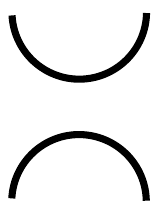}}} $
\caption{$L_0$.} \label{KBSM:Lzero}
\end{subfigure}
\centering
\begin{subfigure}{.25\textwidth}
\centering
$\vcenter{\hbox{\includegraphics[scale = .45]{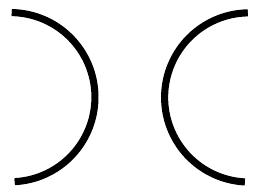}}} $
\caption{$L_{\infty}$.} \label{KBSM:Linfinity}
\end{subfigure}
\caption{The skein triple.} \label{KBSM:skeintriple}
\end{figure}
When the manifold is a disc, the free $R$-module can be equipped with an algebra structure that leads to the classical Temperley-Lieb algebra $\mathit{TL}_n$ (see \cite{PBIMW}, for example). Definition \ref{bilinearformdef} introduces a bilinear form on $\mathit{TL}_n$, and the notion of a Gram matrix.

Physicists  N. Temperley and E. Lieb in \cite{TemperleyLiebTLn} used the Temperley-Lieb algebra to study $2$-dimensional Potts models.
R. J. Baxter in \cite{BaxterExactlysolved} presented the first formal definition of this algebra. Jones independently introduced $TL_n$ in \cite{JonesIndexsubfactors} while working on von Neumann algebras.

\begin{definition} \label{TEMPERLEY:tln}
Let $R$ be a commutative ring with unity and $d \in R$. Let $n \in \mathbb{N}$ be fixed, then the $n^{th}$ \textbf{Temperley-Lieb algebra}, $\mathit{TL}_n$, is defined to be the unital associative algebra over R with generators $ e_1, \dots , e_{n-1}$, identity element $1_n$, 
and relations:
\begin{enumerate}
    \item $e_i e_{j} e_i = e_i \text{ for } |i-j|=1$
    \item $e_i e_j = e_j e_i \text{ for } |i-j|>1$
    \item $e_i^2 = d e_i$
\end{enumerate}
\end{definition}

 L. H.  Kauffman in \cite{KauffmanBracket}, motivated by utilizing the Kauffman bracket, considered the Temperley-Lieb algebra over $R=\mathbb{Z}[A^{\pm1}]$, where $A$ is an indeterminate and $d = -A^2 -A^{-2}$. He then constructed a graphical interpretation using tangles. 

An \textbf{$n$-tangle} is a rectangular shaped disk with $n$ marked boundary points on the left and $n$ marked boundary points on the right. Kauffman's graphical interpretation of the Temperley-Lieb algebra is obtained from the basis of crossingless tangles. The identity element corresponds to an $n$-tangle with $n$ parallel arcs in which each $i^{\mathit{th}}$ point on the left is connected to the $i^{\mathit{th}}$ point on the right, and each $e_i$ corresponds to an $n$-tangle
that has two arcs, each connected to the $i^{\mathit{th}}$ and $(i+1)^{\mathit{th}}$ point on the left  and right, respectively. An illustration is given in Fig. \ref{TEMPERLEY:graphical}. For simplicity we will label an arc by $n$ to denote $n$ parallel arcs as shown in Fig. \ref{TEMPERLEY:fig1}.

\begin{figure}[ht]
\centering
\begin{subfigure}{.49\textwidth}
\centering
$\vcenter{\hbox{
\begin{overpic}[scale = 2]{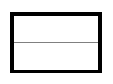}
\put(51, 42){$n$}
\end{overpic} }} $
\caption{Identity element.} \label{TEMPERLEY:fig1}
\end{subfigure}
\begin{subfigure}{.49\textwidth}
\centering
$\vcenter{\hbox{
\begin{overpic}[scale = 2]{tnen}
\put(33, 23){$n-i-1$}
\put(45, 56){$i-1$}
\end{overpic} }} $
\caption{$e_i$.} \label{TEMPERLEY:fig2}
\end{subfigure}
\caption{A graphical interpretation of the generators of $\mathit{TL}_n$.}\label{TEMPERLEY:graphical}
\end{figure}

\begin{definition}
The \textbf{$n$-tangle algebra} is an $R$-module with basis elements consisting of $n$-tangles. The multiplication of two $n$-tangles is defined by identifying the right side of the first $n$-tangle to the left side of the second $n$-tangle, while respecting the boundary points and by letting any resulting trivial curve be denoted by $d$, see Fig. \ref{JONESWENZL:mutip} for an illustrative example. Kauffman's diagrammatic interpretation of the Temperley-Lieb algebra, also known as the \textbf{diagrammatic algebra}, is a subalgebra of the $n$-tangle algebra. It is generated by tangles with no crossings where homotopically trivial curves are denoted by $d \in R$.  
\end{definition}

\begin{figure}[ht]
$$e_2 e_2 =\vcenter{\hbox{\begin{overpic}[scale=1.5]{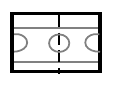}
\end{overpic}}} = d \vcenter{\hbox{\begin{overpic}[scale=1.5]{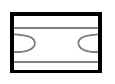}
\end{overpic}}}= d e_2.$$
\caption{An illustration of  multiplication.}\label{JONESWENZL:mutip}
\end{figure}

\begin{theorem}\cite{KauffmanBracket}
The diagrammatic algebra is isomorphic to $TL_n$ and can be thought of as a diagrammatic interpretation of it.
\end{theorem}

\section{Determinant of Type A}
\begin{definition}\
\label{bilinearformdef}
Let $R=\mathbb{Z}[A^{\pm1}]$, $d= -A^2-A^{-2}$, and consider the disk $D^2$ with $2n$ framed points on its boundary. Let $\mathcal{C}_n^A = \{c_1, c_2,\ldots, c_{c_n}\}$ be the set of all diagrams with crossingless connections, up to ambient isotopy, between the $2n$ framed points in $D^2$. 
Define a bilinear form $\langle \ , \ \rangle_A$ in the following way: $$\langle \ , \ \rangle_A : \mathcal{S}_{2,\infty}(D^2 \times I, \{x_i\}_1^{2n}) \times \mathcal{S}_{2,\infty}(D^2 \times I, \{x_i\}_1^{2n}) \longrightarrow R.$$

Let $c_i, c_j \in \mathcal{C}_n^A.$ Glue $c_i$ with the inversion of $c_j$ along the marked circle, respecting the labels of the framed points. The resulting picture is that of a disk with disjoint null homotopic circles. Thus, we define, $\langle c_i , c_j\rangle_A = d^m$ where $m$ denotes the number of these circles.  Figure \ref{fig:ann2nbndpts} illustrates an example of the bilinear form when $n = 4$.

\begin{figure}[ht]
\centering
$\left\langle \vcenter{\hbox{\begin{overpic}[scale = .3]{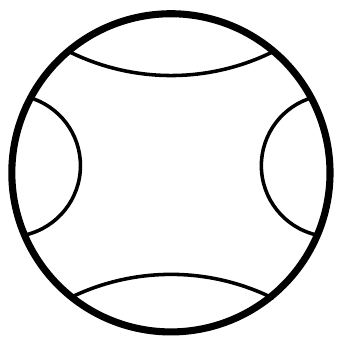}
\end{overpic} }}, \vcenter{\hbox{\begin{overpic}[scale = .3]{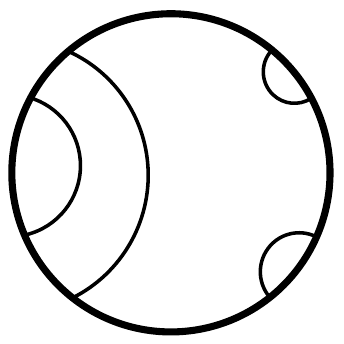}
\end{overpic} }} \right\rangle_A = \vcenter{\hbox{\begin{overpic}[scale = .4]{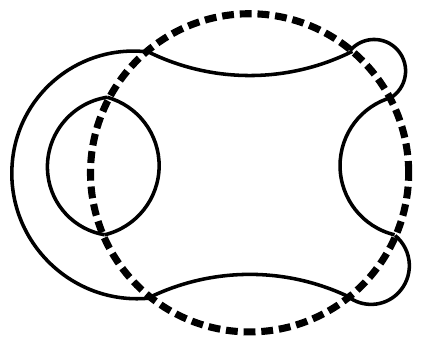}
\end{overpic} }}$
    \caption{The bilinear form on two elements in $\mathit{TL}_4$. Here, the result is $d^2$.}
    \label{fig:ann2nbndpts}
    
    \end{figure}

The \textbf{Gram matrix of type} $\boldsymbol{A}$ is defined as $G_n^{A} = ( \langle a_i , a_j\rangle_A ) _{1 \leq i,j \leq C_n} $. Its determinant $D_n^{A}$ is called the \textbf{Gram determinant of type} $\boldsymbol{A}$.

\end{definition}

\begin{example}\label{GRAMA:G3}
The following table shows the Gram matrix of type $A$ for $n=3$, $G_3^A$. It can be seen that $D_3^A = (d^2-1)^4\cdot d^4 \cdot (d^2-2d)$.

\vspace{0.25 in}
\begin{table}[H]
\centering
\begin{tabular}{ c|||c|c|c|c|c|c||c|c|c|c||c}
 \huge{$\langle \ , \ \rangle_A$} & $\vcenter{\hbox{\begin{overpic}[scale = .2]{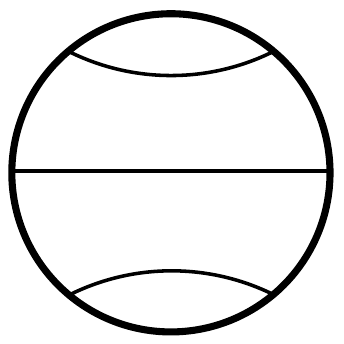}
 \put(12,36){$b_1$}
\end{overpic} }}$ 
& $\vcenter{\hbox{\begin{overpic}[scale = .2]{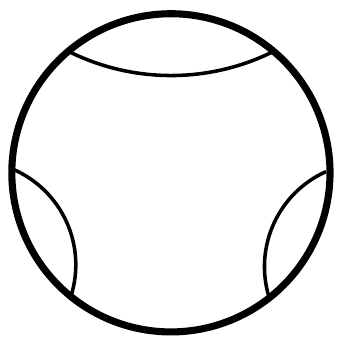}
 \put(12,36){$b_2$}
\end{overpic} }}$ 
& $\vcenter{\hbox{\begin{overpic}[scale = .2]{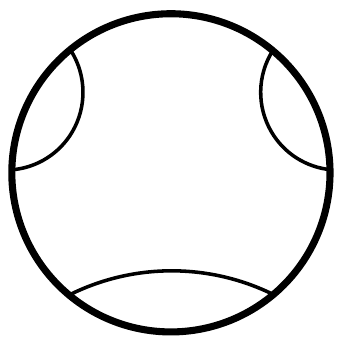}
 \put(12, 36){$b_3$}
\end{overpic} }}$ 
& $\vcenter{\hbox{\begin{overpic}[scale = .2]{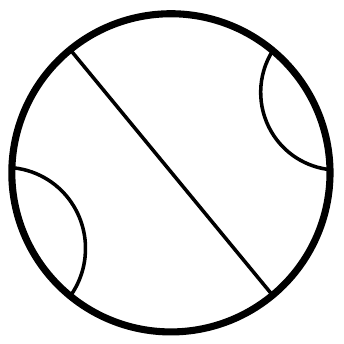}
 \put(12, 36){$b_4$}
\end{overpic} }}$ 
& $\vcenter{\hbox{\begin{overpic}[scale = .2]{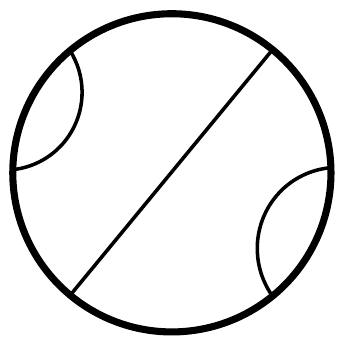}
 \put(12, 36){$b_5$}
\end{overpic} }}$ \\
\hline \hline \hline 	
$\vcenter{\hbox{\includegraphics[scale = .13, height = 1.2cm]{T1.pdf}}}$        
& $ d^3$ & $ d^2$ & $d^2$ & $d$ & $d$ \\ \hline 
$\vcenter{\hbox{\includegraphics[scale = .13, height = 1.2cm]{T2.pdf}}}$          
& $ d^2$ & $ d^3$ & $d$ & $d^2$ & $d^2$ \\ \hline 
$\vcenter{\hbox{\includegraphics[scale = .13, height = 1.2cm]{T3.pdf}}}$ 
& $ d^2$ & $ d$ & $d^3$ & $d^2$ & $d^2$ \\ \hline 
$\vcenter{\hbox{\includegraphics[scale = .13, height = 1.2cm]{T4.pdf}}}$  
& $ d$ & $ d^2$ & $d^2$ & $d^3$ & $d$ \\ \hline 
$\vcenter{\hbox{\includegraphics[scale = .13, height = 1.2cm]{T5.pdf}}}$  
& $ d$ & $ d^2$ & $d^2$ & $d$ & $d^3$  \\ \hline 
\end{tabular}
 \caption{The Gram matrix $G_3^A$.}
 \end{table}

\end{example}

We now recall the definition of the Chebyshev polynomials of the second kind. Let $d$ be the free variable and denote the polynomial by $\Delta_n$. This polynomial is given by:
\begin{align*}
\Delta_0(d) &= 1, \\
\Delta_1(d) &= d,\\
\Delta_n(d) &= d \cdot \Delta_{n-1}(d) - \Delta_{n-2}(d).
\end{align*}

Therefore, from Example \ref{GRAMA:G3} we get $D_3^A = \Delta_1^4 \cdot \Delta_2^4 \cdot \Delta_3$.

\ 

A closed formula for the calculation of this determinant was found by Bruce Westbury \cite{Wes} and Philippe Di Francesco \cite{DiF}. One of the tools used by Di Francesco was the Gram-Schmidt orthogonalization together with some combinatorial methods. In \cite{Cai}, Xuanting Cai used Jones-Wenzl idempotents to significantly shorten the proof in \cite{DiF}. For a more detailed discussion about this formula and its proof see Lecture 17 in \cite{BIMP}.

\begin{theorem}\cite{Wes, DiF, Cai}\label{GRAMA:maintheorem}
Let $R = \mathbb{Z}[A^{\pm 1}].$ Then,
\begin{align*}
D_n^{A}(d) = \prod\limits_{i=1}^{n} \left(\frac{\Delta _i}{\Delta_{i-1}} \right)^{\alpha _i},
\end{align*}

\begin{align*}
\text{where }  \Delta_i = (-1)^i \frac{A^{2i+2}-A^{-2i-2}}{A^2-A^{-2}} \text{ and } \alpha _i = \binom{2n}{n-i} - \binom{2n}{n-i-1}.
\end{align*}
\end{theorem}

\section{Generalization of Type A}\label{GRAMA:GenA} \

\noindent Here we consider a different bilinear form that results in a new type of Gram determinant \cite{BIMP}. Similarly as with the determinant of type $A$, the basis elements are the crossingless connections on $2n$ boundary points on a disk. The construction is presented below. 

\begin{definition}
Let the disk $D^2$, with $2n$ marked points on its boundary, be considered as a rectangle with $n$ points on the top edge and $n$ points on the bottom edge. Define a bilinear form $\langle \ , \ \rangle_{A^{gen}}$ in the following way: 
$$\langle \ , \ \rangle_{A^{gen}}: \mathcal{S}_{2,\infty}(D^2 \times I, \{x_i\}_1^{2n}) \times \mathcal{S}_{2,\infty}(D^2 \times I, \{x_i\}_1^{2n}) \longrightarrow \mathbb{Z}[d, z].$$
We view the elements of $\mathit{TL}_n$ vertically and multiplying from top to bottom. For $a_i, a_j \in \mathcal{A}_n$, glue $a_i$ with the reflection about the horizontal axis of $a_j$, which is denoted by $\overline{a_j}$, such that the bottom edge of $a_i$ is identified with the top edge of $\overline{a_j}$. Connect the marked points on the top edge of $a_i$ with those on the bottom edge of $\overline{a_j}$, in the annulus, respecting the ordering of the marked points (see Figure \ref{GRAMA:Genex}). Recall that the described operation is the same as taking the trace of the image of the bilinear form.

\begin{figure}[ht]
    \centering
  $$ \langle a_i, a_j \rangle_{A^{gen}} = \vcenter{\hbox{\begin{overpic}[scale = .4]{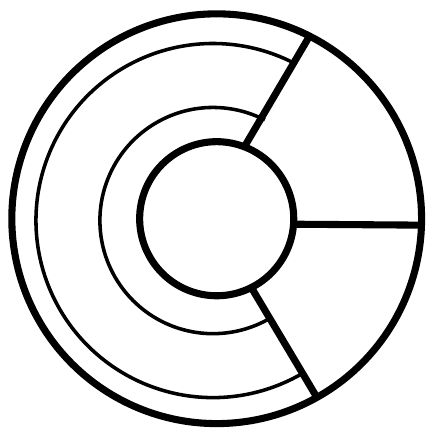}
\put(49.5, 68){$\cdot$}
\put(48, 65){$\cdot$}
\put(46.5, 62){$\cdot$}
\put(60,55){$a_i$}
\put(60,25){$\overline{a_j}$}
\put(49.5, 9.5){$\cdot$}
\put(48, 12.5){$\cdot$}
\put(46.5, 15.5){$\cdot$}
\end{overpic} }}= d^kz^h. $$
    \caption{The bilinear form of generalized type $A$.}
    \label{GRAMA:Genex}
\end{figure}

The result is an annulus with two types of disjoint circles, homotopically trivial and non - trivial. Thus, we define $\langle a_i , a_j\rangle_{A^{gen}} = d^kz^m$ where $k$ and $m$ denote the number of these circles respectively.
We define the \textbf{Gram matrix of generalized type} $\boldsymbol{A}$ as $G_n^{A^{\mathit{gen}}} = ( \langle a_i , a_j\rangle_{A^{gen}} ) _{1 \leq i,j \leq C_n}$, and denote its determinant by  $D_n^{A^{\mathit{gen}}}$.

\end{definition}

\begin{example}
 \label{GRAMA:GENAEX1}
Consider the basis of Catalan connections on six boundary points. Figure \ref{GRAMA:innerprod} illustrates the calculation of $\langle a_4, a_5 \rangle_{A^{gen}}$. The matrix $G_3^{A^{\mathit{gen}}}$ is computed in Table \ref{gramatable}, along with its determinant.

\begin{figure}[ht]
    \centering
  $$\langle a_4, a_5 \rangle_{A^{gen}} = \left\langle \vcenter{\hbox{\begin{overpic}[unit=1mm, scale = .25]{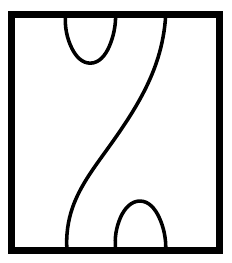}
 \put(5,14){$a_4$} 
\end{overpic} }},
\vcenter{\hbox{\begin{overpic}[unit=1mm, scale = .25]{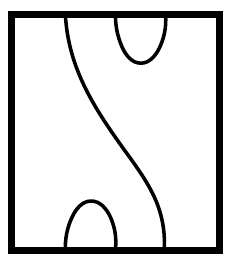}
 \put(5,14){$a_5$}
\end{overpic} }}
\right\rangle =
\vcenter{\hbox{\begin{overpic}[scale = .3]{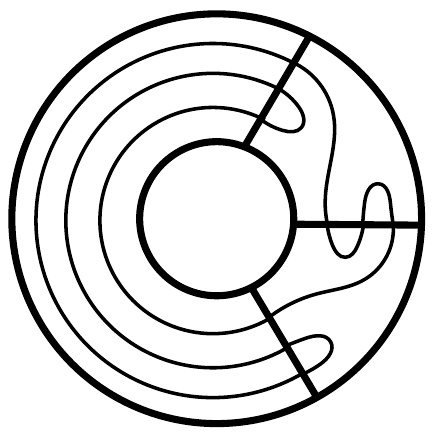}
\put(56,50){$a_4$}
\put(56,5){$\overline{a_5}$}
\end{overpic} }}= z $$
    \caption{The bilinear form of generalized type $A$.}
    \label{GRAMA:innerprod}
\end{figure}

\begin{table}[ht]
\centering
\begin{tabular}{ c|||c| c|  c| c |c| }
 \large{$ \langle \ , \ \rangle_{A^{gen}}$} & $\vcenter{\hbox{\includegraphics[scale = .3, height = 1cm]{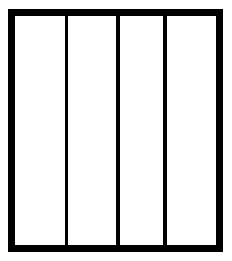}}}$ 
& $\vcenter{\hbox{\includegraphics[scale = .3,height = 1cm]{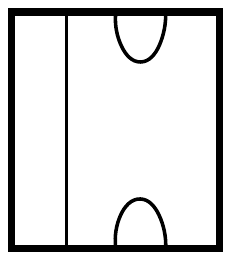}}}$
& $\vcenter{\hbox{\includegraphics[scale = .3,height = 1cm]{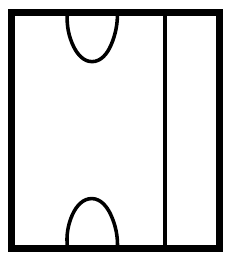}}}$
& $\vcenter{\hbox{\includegraphics[scale = .3,height = 1cm]{R4.pdf}}}$
& $\vcenter{\hbox{\includegraphics[scale = .3,height = 1cm]{R5.pdf}}}$
\\
\hline \hline \hline 

$\vcenter{\hbox{\includegraphics[scale = .3, height = 1cm]{R1.pdf}}}$       
& $ z^3$ & $ dz$ & $dz$ & $z$ & $z$ \\
\hline

$\vcenter{\hbox{\includegraphics[scale = .3, height = 1cm]{R2.pdf}}}$         
& $dz$
& $d^2z$ & $z$ & $dz$ & $dz$ \\
\hline
 
$\vcenter{\hbox{\includegraphics[scale = .3, height = 1cm]{R3.pdf}}}$          
& $dz$
& $z$ & $d^2z$ & $dz$ & $dz$ \\
\hline

$\vcenter{\hbox{\includegraphics[scale = .3, height = 1cm]{R4.pdf}}}$ 
& $ z$ & $ dz$ & $dz$ & $d^2z$ & $z$\\ 
\hline

 $\vcenter{\hbox{\includegraphics[scale = .3, height = 1cm]{R5.pdf}}}$& $ z$ & $dz$ & $dz$ & $z$ & $d^2z$  \\ \hline 
\end{tabular}
\caption{The Gram matrix $G_3^{A^{\mathit{gen}}}$.}
\label{gramatable}
\end{table}

We calculate and factor the determinant: $$D_3^{A^{\mathit{gen}}} = (d^2-1)^4z^5(z^2-2).$$

\end{example}

\section{Determinant of Type B}\label{typeB}\

\noindent It is possible to construct different bilinear forms by changing the ambient surface of the Kauffman bracket skein module. For example, we can explore the notion of a Gram determinant using the annulus and the M\"obius band. Paul Martin and Hubert Saleur were the first to consider the type $B$ Gram determinant. Their work focuses on applications to statistical mechanics, and they employ representation theory to prove their determinant in \cite{MS1, MS2}. For a historical overview about the development of this type of Gram determinants, the reader is referred to \cite{PBIMW, IM}.

\ 

Consider an annulus with $2n$ marked points along the outer boundary; see Figure~\ref{GRAMB:figann2nbndpts}.

\begin{figure}[ht]
\centering
$$\vcenter{\hbox{\begin{overpic}[unit=1mm, scale = .65]{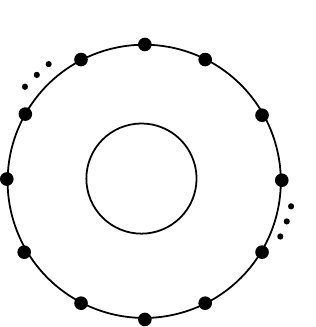}

\put(14.5,33){$a_1$}
\put(23,31){$a_2$}
\put(30, 6){$a_i$}
\put(0,31.5){$a_{2n-1}$}

\end{overpic} }}$$
    \caption{An annulus with 2n boundary points.}
    \label{GRAMB:figann2nbndpts}
    \end{figure}

Similarly as the type $A$ determinant, we can connect these $2n$ points to form crossingless connections in the annulus. Denote by $B_n$ the set of these diagrams. It can be shown that there are $\binom{2n}{n}$ such diagrams. See Figure \ref{fig:Basis2}, which illustrates $B_2$.

\begin{figure}[ht]
\[  \begin{minipage}{.5in} \includegraphics[width=\textwidth]{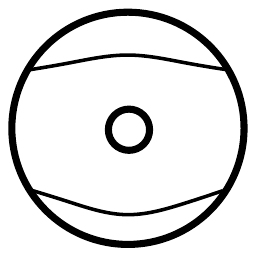} \vspace{-15pt} \[\pmb b_1\] \end{minipage} 
               \qquad
        \begin{minipage}{.5in}\includegraphics[width=\textwidth]{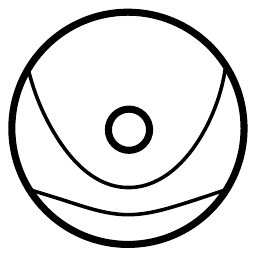} \vspace{-15pt} \[\pmb b_2\] \end{minipage}
         \qquad
        \begin{minipage}{.5in}\includegraphics[width=\textwidth]{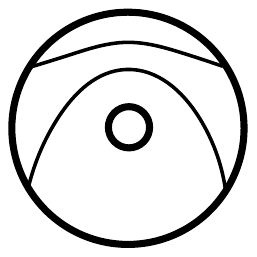} \vspace{-15pt} \[ \pmb b_3\]\end{minipage}
        \qquad
        \begin{minipage}{.5in}\includegraphics[width=\textwidth]{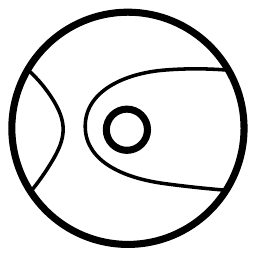} \vspace{-15pt} \[ \pmb b_4\]\end{minipage}
        \qquad
        \begin{minipage}{.5in}\includegraphics[width=\textwidth]{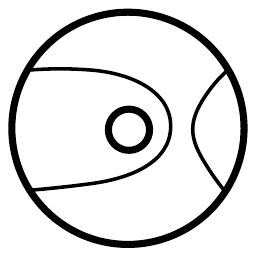} \vspace{-15pt} \[ \pmb b_5\]\end{minipage}   
        \qquad
        \begin{minipage}{.5in}\includegraphics[width=\textwidth]{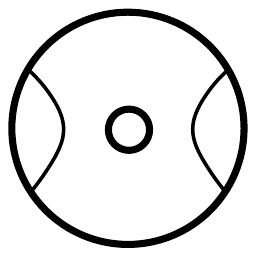} \vspace{-15pt} \[ \pmb b_6\]\end{minipage} 
        \hfill
        \]
    \caption{The set $B_2$.}
    \label{fig:Basis2}
\end{figure}

The bilinear form for Type $B$ is defined on these annular diagrams as follows: 
\begin{definition}\
Let $R = \mathbb{Z}[A^{\pm1}]$, $d=-A^2-A^{-2}$, and let $A^2$ be an annulus with $2n$ marked points on its outer boundary. Let $B_n = \{b_1, b_2, \ldots, b_{\binom{2n}{n}} \}$ be the set of all diagrams of crossingless connection between these $2n$ points. Define the type $B$ bilinear form $\langle \ \ , \ \rangle$ in the following way: $$\langle \ \ , \ \rangle : \mathcal{S}_{2,\infty}(A^2 \times I, \{x_i\}_1^{2n}; R, A) \times \mathcal{S}_{2,\infty}(A^2 \times I, \{x_i\}_1^{2n}; R, A) \longrightarrow R[z].$$

Given $b_i, b_j \in B_n$, glue $b_i$ with the inversion of $b_j$ along the marked circle, respecting the labels of the marked points. The resulting picture has disjoint circles, which are either homotopically non-trivial or null homotopic. Then, $\langle b_i , b_j\rangle = z^k d^m$, where $k$ and $m$ denote the number of these circles, respectively.

The Gram matrix of type $B$ is defined as $G_n^{B} = (\langle b_i , b_j\rangle) _{1 \leq i, j \leq \binom{2n}{n}}.$ Its determinant $D_n^B$ is called the Gram determinant of type $B$.

\end{definition}

\begin{example}
The bilinear form $\langle b_5, b_6 \rangle$ is illustrated in Figure \ref{GRAMB:innerprod}.
\begin{figure}[ht]
\centering
$$\left\langle \vcenter{\hbox{\begin{overpic}[unit=1mm, scale = 0.4]{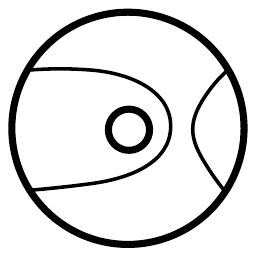}
\end{overpic} }}, \vcenter{\hbox{\begin{overpic}[unit=1mm, scale = .4]{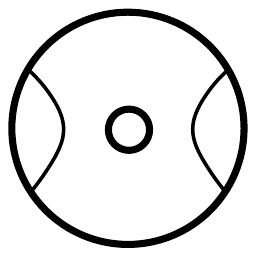}
\end{overpic} }} \right\rangle = \vcenter{\hbox{\begin{overpic}[unit=1mm, scale = .4]{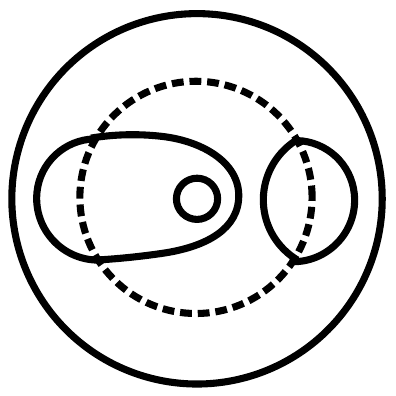}
\put(20,18){$d$}
\put(4, 18){$z$}
\end{overpic} }} = dz$$
    \caption{The inner product of two elements in $B_4$. In this case, the result is $dz$.}
    \label{GRAMB:innerprod}
    \end{figure}
\end{example}

\begin{example} 
Table \ref{typebmatrix} illustrates the Gram matrix $G_2^B$. We calculate the Gram determinant for the collection $B_2$.

\begin{table}[ht]
\centering
\begin{tabular}{ c|||c|c|c|c|c|c|c||c|c|c|c|c||c}
 \huge{$ \langle \ , \ \rangle$} & $\vcenter{\hbox{\includegraphics[scale = .25, height = 1cm]{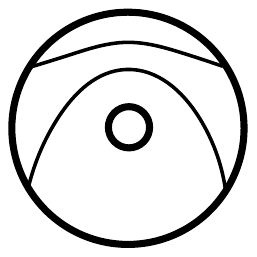}}}$ 
& $\vcenter{\hbox{\includegraphics[scale = .25,height = 1cm]{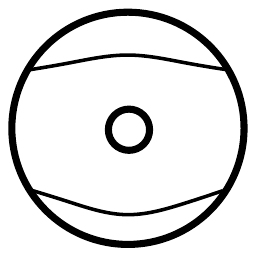}}}$
& $\vcenter{\hbox{\includegraphics[scale = .25,height = 1cm]{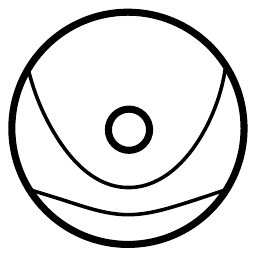}}}$
& $\vcenter{\hbox{\includegraphics[scale = .25,height = 1cm]{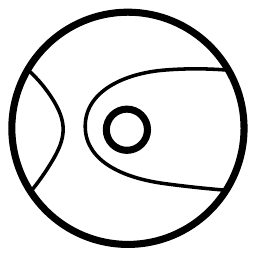}}}$
& $\vcenter{\hbox{\includegraphics[scale = .25,height = 1cm]{b6.pdf}}}$
& $\vcenter{\hbox{\includegraphics[scale = .25,height = 1cm]{b5.pdf}}}$\\
\hline \hline \hline 	
$\vcenter{\hbox{\includegraphics[scale = .25, height = 1cm]{b3.pdf}}}$        
& $ d^2$ & $ dz$ & $z^2$ & $z$ & $d$ & $z$ \\ \hline 
$\vcenter{\hbox{\includegraphics[scale = .25, height = 1cm]{b1.pdf}}}$          
& $ dz$ & $ d^2$ & $dz$ & $d$ & $z$ & $d$\\ \hline 
$\vcenter{\hbox{\includegraphics[scale = .25, height = 1cm]{b2.pdf}}}$ 
& $z^2$ & $ dz$ & $d^2$ & $z$ & $d$ & $z$ \\ \hline 
$\vcenter{\hbox{\includegraphics[scale = .25, height = 1cm]{b4.pdf}}}$  
& $ z$ & $ d$ & $z$ & $d^2$ & $dz$ & $z^2$ \\ \hline 
$\vcenter{\hbox{\includegraphics[scale = .25, height = 1cm]{b6.pdf}}}$  
& $d$ & $z$ & $d$ & $dz$ & $d^2$ & $dz$  \\ \hline 
$\vcenter{\hbox{\includegraphics[scale = .25, height = 1cm]{b5.pdf}}}$  
& $z$ & $d$ & $z$ & $z^2$ & $dz$ & $d^2$  \\ \hline 
\end{tabular}
\caption{The Gram matrix $G_2^B$.}
\label{typebmatrix}
\end{table}

We calculate the determinant of this matrix to be: 
\begin{align*}
D_2^B(d,z) &= -(d-z)^4(d+z)^4(-2+d^2+z)(-2+d^2+z)\\
& =(d^2-z^2)^4((d^2-2)^2-z^2).
\end{align*}
\label{ex182}
\end{example}

Recall that $T_n$ denotes the Chebyshev polynomial of the first kind defined recursively by the equation: $T_{n+1}(d) = d\cdot T_n(d) - T_{n-1}(d)$, with the initial conditions $T_0(d) = 2$ a,nd $T_1(d) = d$. The determinant given above can be expressed in terms of Chebyshev polynomials. In general, it is possible to rewrite the Gram determinant of type $B$ in terms of Chebyshev polynomials of the first kind for all $n$. 

\ 

The following closed formula was proven by Q. Chen and J. Przytycki and the proof involves the creation of a linear map on the basis $B_{n, 0}$ that uses the lollipop method to decorate the inner boundary component with the Jones-Wenzl idempotent. Then, the proof proceeds by showing that the image of the linear map under the basis is a subspace of dimension $\binom{2n}{n} - \binom{2n}{n-k}$ \cite{Ch-P}.

\begin{theorem}\cite{Ch-P,MS2} \label{CH-P}
Let $R = \mathbb{Z}[A^{\pm1}]$, then $D_n^{B}=  \prod \limits_{i=1}^n(T_i(d)^2-z^2)^{\binom{2n}{n-i}}.$
\label{thrm181}
\end{theorem}

\begin{example}\label{example:mbn} An example of the bilinear form on two elements in $Mb_4$ is given below, where $\vcenter{\hbox{\begin{overpic}[scale = .2]{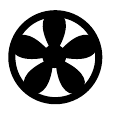}
\end{overpic}}}$ denotes a crosscap usually denoted by $\vcenter{\hbox{\begin{overpic}[scale = 1]{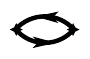} \end{overpic}}}.$
$$\left\langle \vcenter{\hbox{\begin{overpic}[scale = .2]{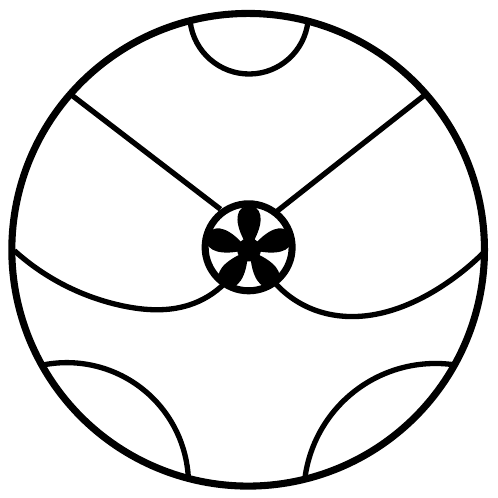}
\end{overpic}}}, \vcenter{\hbox{\begin{overpic}[scale = .2]{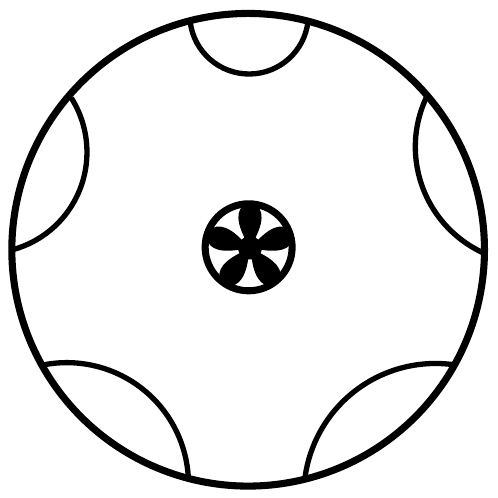}
\end{overpic}}} \right\rangle_{\mathit{Mb}} = \vcenter{\hbox{\begin{overpic}[scale = .2]{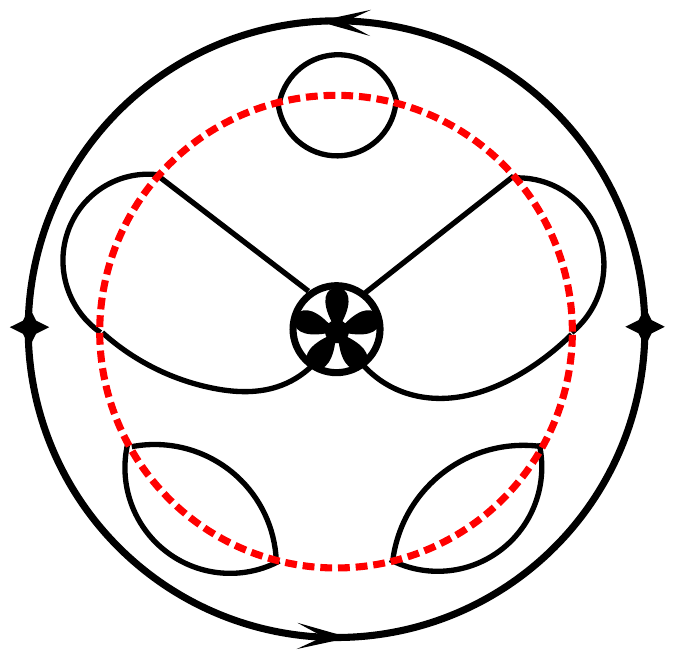}
\end{overpic} }} = d^4. $$
\end{example}

\section{Determinant of Type Mb}\

\noindent This Gram determinant originates from the study of crossingless connections on a M\"obius band. The basis of the relative Kauffman bracket skein module of the twisted $I$-bundle of the M\"obius band is infinite. The Gram matrix of type Mb is created from a finite sub-collection of this basis where we restrict to basis elements with no $z$ and $x$ curves. The bilinear form is defined through the identification of two M\"obius bands along their boundaries. This determinant is given in Definition \ref{TypeMB}.

\begin{definition}\label{TypeMB} 
Let $\mathit{Mb}_n = \{ m_1, \dots, m_{\sum_{k=0}^{n} \binom{2n}{k}}\}$  be the set of all diagrams of crossingless connections between $2n$ marked points on the boundary of the M\"obius band $\mathit{Mb} \ \hat{\times} \ \{0\}$ in $\mathit{Mb} \  \hat{\times} \ I$ with no simple closed curves. 
Define a bilinear form $\langle \ , \ \rangle_{\mathit{Mb}}$ on the elements of $\mathit{Mb}_n$ as follows:
$$ \langle \ , \ \rangle_{\mathit{Mb}} : \mathcal{S}_{2,\infty}(\mathit{Mb}\  \hat \times \ I, \{x_i\}_1^{2n}) \times \mathcal{S}_{2,\infty}(\mathit{Mb}\ \hat \times \  I, \{x_i\}_1^{2n}) \longrightarrow \mathbb{Z}[d,w,x,y,z].$$

Given $m_i, m_j \in \textit{Mb}_n$, identify the boundary component of $m_i$ with that of the inversion of $m_j$, respecting the labels of the marked points. The result is an element in $Kb \ \hat{\times} \ I$ containing only disjoint simple closed curves. The five homotopically distinct simple closed curves in the Klein bottle, including the homotopically trivial curve, are denoted by $x, y, z, w, d$ as illustrated in Figure \ref{Klein}. Then, $\langle m_i , m_j\rangle_{\mathit{Mb}} \coloneqq  d^mx^ny^kz^lw^h$ where $m,n,k,l$ and $h$ denote the number of these curves, respectively.

The Gram matrix of type $\mathit{Mb}$ is defined as $\ G_n^{\mathit{Mb}} = (\langle m_i , m_j\rangle_{\mathit{Mb}})_{1 \leq i, j \leq \sum_{k=0}^{n} \binom{2n}{k}}$ and its determinant $D_n^{\mathit{Mb}}$ is called the Gram determinant of type $\mathit{Mb}$.

\begin{figure}[h]
    \centering
    $$ \vcenter{\hbox{
\begin{overpic}[scale=.4]{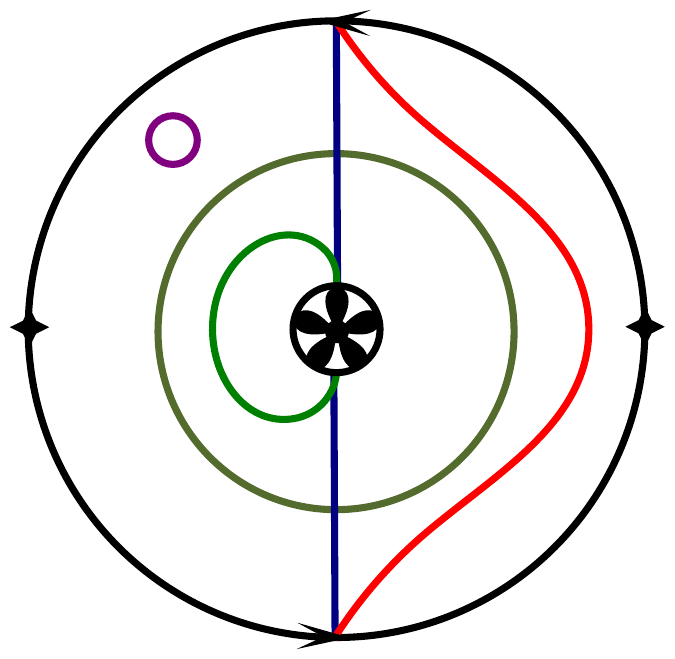}
\put(38, 106){$d$}
\put(22, 70){$z$}
\put(43, 61){$x$}
\put(90, 24){$y$}
\put(54, 105){$w$}
\end{overpic} }}$$

     \caption{Klein bottle and its five homotopically distinct simple closed curve; the curve $w$ intersects both crosscaps, $y$ intersects the outer crosscap, $x$ intersects the inner crosscap, $z$ is a nontrivial curve that does not intersect the crosscaps, and $d$ is the trivial curve.}
    \label{Klein}
\end{figure}
\end{definition}
Qi Chen conjectured the following result for the Gram determinant of type $\mathit{Mb}$ and verified the conjecture for $n\leq 4$. Some work supporting this conjecture can be found in \cite{BIMP, BIMP2, PBIMW}. 

\begin{conjecture}[Chen] \label{Qi}\ \cite{Che}

Let $R = \mathbb{Z}[A^{\pm 1},w,x,y,z]$, $D_{n,i} = \prod\limits_{k=1+i}^n (T_{2k}(d)-2)^{\binom{2n}{n-k}}$, and $i$ represents the number of curves passing through the crosscap. Then
the Gram determinant of type {\it Mb} for $n \geq 1$, denoted by $D_n^{\mathit{Mb}}$, is:

 \begin{eqnarray*}
D^{\mathit{Mb}}_n(d,w,x,y,z) &=& \prod_{k=1}^n (T_k(d)+(-1)^kz)^{\binom{2n}{n-k} } \\
& & \prod\limits_{ \substack{k=1 \\ k\text{ odd }}}^n
\left((T_k(d) - (-1)^k z)T_k(w) -2xy\right)^{\binom{2n}{n-k}} \\
& & \prod\limits_{ \substack{k=1 \\ k\text{ even }}}^n
\left((T_k(d) - (-1)^kz)T_k(w)-2(2-z)\right)^{\binom{2n}{n-k}} \\
& & \prod_{i=1}^{n-1} D_{n,i}.
\end{eqnarray*}
\end{conjecture}

\begin{proposition}\cite{BIMP2}\label{GRAMMB:Prop1}
$D^{\mathit{Mb}}_n$ is divisible by $(d-z)^{\binom{2n}{n-1}}$.
\end{proposition}

\begin{proposition}\label{GRAMMB:Prop2}\cite{BIMP}
$D^{\mathit{Mb}}_n$  is divisible by $(w(d+z)-2xy)^{\binom{2n}{n-1}}$.
\end{proposition}

\begin{proposition}\label{GRAMB}\cite{PBIMW}
$D_n^{\mathit{Mb}}$ is divisible by $((d^2-2-z)(w^2-2)-2(2-z))^{\binom{2n}{n-2}}$.
\end{proposition}

\section{Determinant of Type $(Mb)_1$} \label{GramMB1}
The Gram matrix of type $\mathit{Mb}$ was created by taking a finite sub-collection of the basis of the relative Kauffman bracket skein module of the twisted $I$-bundle of the M\"obius band. However, the number of such elements increase exponentially as $n$ increases. In particular, computing the determinant for $n \geq 5$ has not been achieved due to the size of the matrices.  In this section we present a new Gram determinant created in \cite{IM} by using a sub-collection of $Mb_n$.

\begin{definition}\label{TypeMB1} 
Let $(\mathit{Mb}_{n})_1 = \mathit{Mb}_{n,0} \cup \mathit{Mb}_{n,1}$ where $\mathit{Mb}_{n,0} = \{ m_1, \dots, m_{\binom{2n}{n}}\}$ is the set of all diagrams of crossingless connections between $2n$ marked points on the boundary of $\mathit{Mb} \ \hat{\times} \ \{0\}$ whose arcs do not intersect the crosscap and $\mathit{Mb}_{n,1} = \{ m_1, \dots, m_{\binom{2n}{n-1}}\}$ is the set of all diagrams of crossingless connections between $2n$ marked points on the boundary of  $\mathit{Mb} \ \hat{\times} \ \{0\}$ with exactly one curve intersecting the crosscap. 
Define a bilinear form $\langle \ , \ \rangle_{\mathit{Mb}}$ on the elements of $(\mathit{Mb}_{n})_1$ by using the same bilinear form as type $\mathit{Mb}$, as follows: $$ \langle \ , \ \rangle_{\mathit{Mb}} : \mathcal{S}_{2,\infty}(\mathit{Mb}\  \hat \times \ I, \{x_i\}_1^{2n}) \times \mathcal{S}_{2,\infty}(\mathit{Mb}\ \hat \times \  I, \{x_i\}_1^{2n}) \longrightarrow \mathbb{Z}[d,w,x,y,z].$$

Given $m_i, m_j \in (\mathit{Mb}_{n})_1$, identify the boundary component of $m_i$ with that of the inversion of $m_j$, respecting the labels of the marked points. The result is an element in $Kb \ \hat{\times} \ I$ containing only disjoint simple closed curves. Then $\langle m_i , m_j\rangle_{\mathit{Mb}} :=  d^mx^ny^kz^lw^h$ where $m,n,k,l$ and $h$ denote the number of these curves, respectively.

The Gram matrix of type $(\mathit{Mb})_1$ is defined as  $G_n^{(\mathit{Mb})_1} = (\langle m_i , m_j\rangle_{\mathit{Mb}})_{1 \leq i, j \leq |(\mathit{Mb})_1|}$
and its determinant $D_n^{(\mathit{Mb})_1}$ is called the Gram determinant of type $(\mathit{Mb})_1$.
\end{definition}

\begin{example}\label{exampleMB1} $n=2$ yields the smallest Gram matrix of type $(Mb)_1$ that differs from type $Mb$. The set $(Mb_2)_1$ is illustrated in Figure \ref{fig:BasisM2}.

\begin{figure}[ht]
\[  \begin{minipage}{.5in} \includegraphics[width=\textwidth]{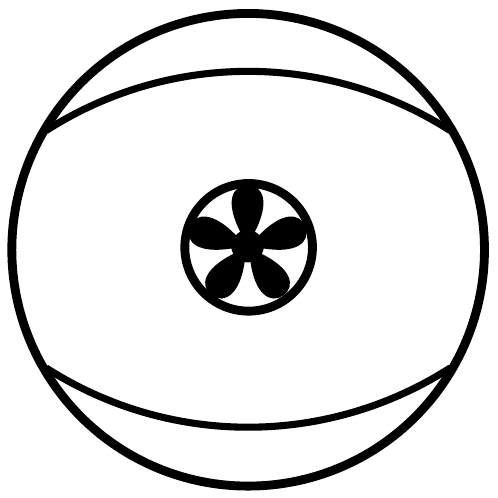} \vspace{-15pt} \[\pmb{(2 \ 4)}\] \end{minipage} 
               \qquad
        \begin{minipage}{.5in}\includegraphics[width=\textwidth]{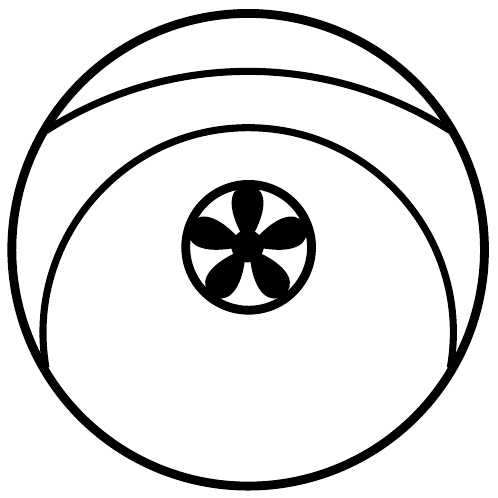} \vspace{-15pt} \[\pmb{(3 \ 4)}\] \end{minipage}
         \qquad
        \begin{minipage}{.5in}\includegraphics[width=\textwidth]{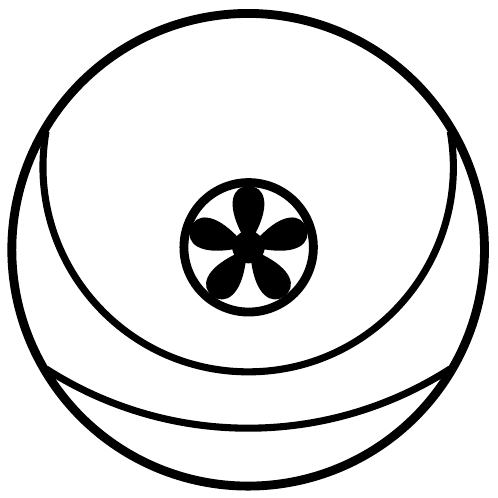} \vspace{-15pt} \[ \pmb{(2 \ 4)}\]\end{minipage}
        \qquad
        \begin{minipage}{.5in}\includegraphics[width=\textwidth]{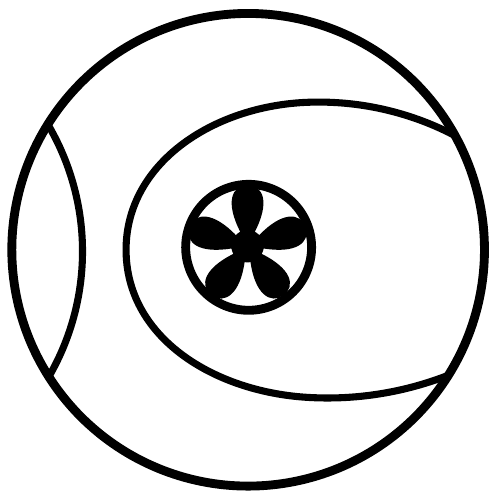} \vspace{-15pt} \[ \pmb{(1 \ 4)}\]\end{minipage}
        \qquad
        \begin{minipage}{.5in}\includegraphics[width=\textwidth]{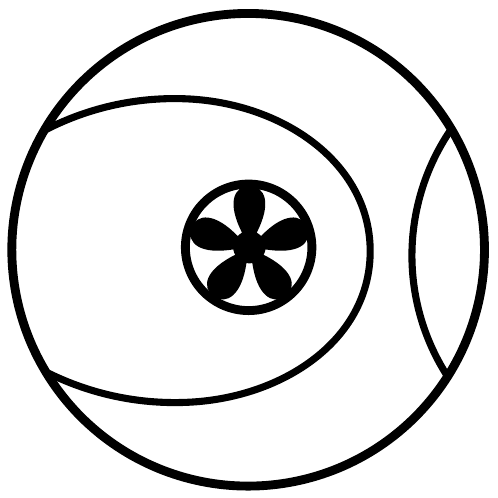} \vspace{-15pt} \[ \pmb{(2 \ 3)}\]\end{minipage}   
        \qquad
        \begin{minipage}{.5in}\includegraphics[width=\textwidth]{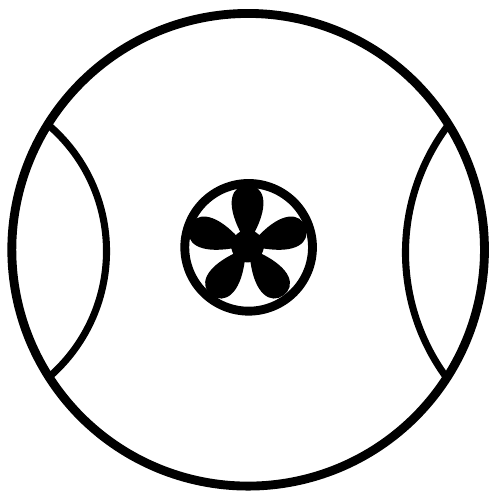} \vspace{-15pt} \[ \pmb{(1 \ 3)}\]\end{minipage} 
        \hfill \]\\
        \begin{minipage}{.5in}\includegraphics[width=\textwidth]{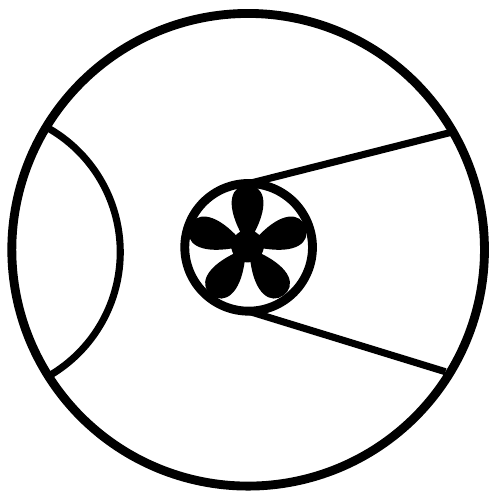} \vspace{-15pt} \[ \pmb{(1)}\]\end{minipage} 
\qquad 
\begin{minipage}{.5in}\includegraphics[width=\textwidth]{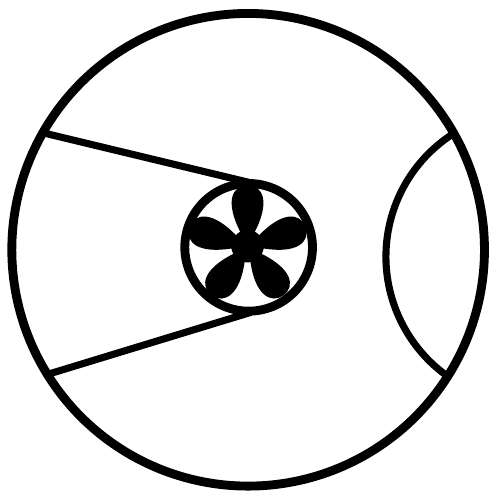} \vspace{-15pt} \[ \pmb{(3)}\]\end{minipage}   
        \qquad
        \begin{minipage}{.5in}\includegraphics[width=\textwidth]{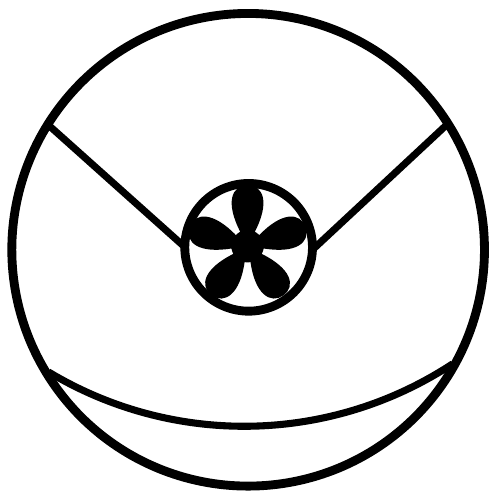} \vspace{-15pt} \[ \pmb{(2)}\]\end{minipage}   
        \qquad
        \begin{minipage}{.5in}\includegraphics[width=\textwidth]{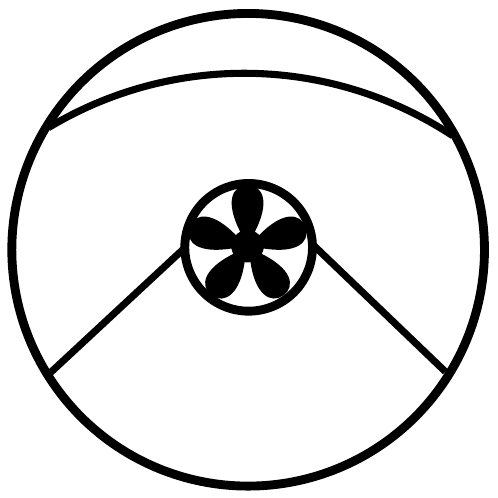} \vspace{-15pt} \[ \pmb{(4)}\]\end{minipage}   
    \caption{The set $(Mb_2)_1$.}
    \label{fig:BasisM2}
\end{figure}

{\centering
\resizebox{\columnwidth}{!}{
\begin{tabular}{ c|||c|c|c|c|c|c||c|c|c|c}
 $\langle \ , \ \rangle$ & $\vcenter{\hbox{\includegraphics[scale = .13, height = 0.7cm]{GMB2_2.pdf}}}$ 
& $\vcenter{\hbox{\includegraphics[scale = .13,height = 0.7cm]{GMB2_1.pdf}}}$
& $\vcenter{\hbox{\includegraphics[scale = .13,height = 0.7cm]{GMB2_3.pdf}}}$
& $\vcenter{\hbox{\includegraphics[scale = .13,height = 0.7cm]{GMB2_4.pdf}}}$
& $\vcenter{\hbox{\includegraphics[scale = .13,height = 0.7cm]{GMB2_6.pdf}}}$
& $\vcenter{\hbox{\includegraphics[scale = .13,height = 0.7cm]{GMB2_5.pdf}}}$
& $\vcenter{\hbox{\includegraphics[scale = .13,height = 0.7cm]{GMB2_10.pdf}}}$
& $\vcenter{\hbox{\includegraphics[scale = .13,height = 0.7cm]{GMB2_8.pdf}}}$
& $\vcenter{\hbox{\includegraphics[scale = .13,height = 0.7cm]{GMB2_9.pdf}}}$
& $\vcenter{\hbox{\includegraphics[scale = .13,height = 0.7cm]{GMB2_7.pdf}}}$ \\
\hline \hline \hline 	
$\vcenter{\hbox{\includegraphics[scale = .13, height = 0.7cm]{GMB2_2.pdf}}}$        
& $ d^2$ & $ dz$ & $z^2$ & $z$ & $d$ & $z$ & $dy$ & $y$ & $yz$ & $ y$  \\ \hline 
$\vcenter{\hbox{\includegraphics[scale = .13, height = 0.7cm]{GMB2_1.pdf}}}$          
& $ dz$ & $ d^2$ & $dz$ & $d$ & $z$ & $d$ & $dy$ & $y$ & $dy$ & $ y$  \\ \hline 
$\vcenter{\hbox{\includegraphics[scale = .13, height = 0.7cm]{GMB2_3.pdf}}}$ 
& $ z^2$ & $ dz$ & $d^2$ & $z$ & $d$ & $z$ & $yz$ & $y$ & $dy$ & $ y$  \\ \hline 
$\vcenter{\hbox{\includegraphics[scale = .13, height = 0.7cm]{GMB2_4.pdf}}}$  
& $ z$ & $ d$ & $z$ & $d^2$ & $dz$ & $z^2$ & $y$ & $yz$ & $y$ & $ dy$ \\ \hline 
$\vcenter{\hbox{\includegraphics[scale = .13, height = 0.7cm]{GMB2_6.pdf}}}$  
& $ d$ & $ z$ & $d$ & $dz$ & $d^2$ & $dz$ & $y$ & $dy$ & $y$ & $ dy$  \\ \hline 
$\vcenter{\hbox{\includegraphics[scale = .13, height = 0.7cm]{GMB2_5.pdf}}}$  
& $ z$ & $ d$ & $z$ & $z^2$ & $dz$ & $d^2$ & $y$ & $dy$ & $y$ & $ yz$  \\ \hline \hline
$\vcenter{\hbox{\includegraphics[scale = .13, height = 0.7cm]{GMB2_10.pdf}}}$  
& $ dx$ & $ dx$ & $xz$ & $x$ & $x$ & $x$ & $dw$ & $w$ & $xy$ & $w$ \\ \hline 
$\vcenter{\hbox{\includegraphics[scale = .13, height = 0.7cm]{GMB2_8.pdf}}}$  
& $x$ & $ x$ & $x$ & $xz$ & $dx$ & $dx$ & $w$ & $dw$ & $w$ & $ xy$  \\ \hline 
$\vcenter{\hbox{\includegraphics[scale = .13, height = 0.7cm]{GMB2_9.pdf}}}$  
& $ xz$ & $ dx$ & $dx$ & $x$ & $x$ & $x$ & $xy$ & $w$ & $dw$ & $ w$  \\ \hline 
$\vcenter{\hbox{\includegraphics[scale = .13, height = 0.7cm]{GMB2_7.pdf}}}$  
& $ x$ & $ x$ & $x$ & $dx$ & $dx$ & $xz$ & $w$ & $xy$ & $w$ & $ dw$ 
\end{tabular}}}
\captionof{table}{The Gram matrix $G_{2}^{(Mb)_1} .$}\label{Grammatrixn2}

\begin{eqnarray*}
D_2^{(\mathit{Mb})_1} &=& (d^2-4) d^2 (d - z)^4 (-2 + d^2 - z) (-2 + d^2 + z) (-d w + 
   2 x y - w z)^4 \\
   &=& (d-z)^4((d^2-2)+z) ((d+z)w-2xy)^4((d^2-2)-z)(d^2(d^2-4)) \\
   &=& (T_1(d)-z)^4(T_2(d)^2-z^2)((d+z)w-2xy)^4(T_4(d)-2).
\end{eqnarray*}
\end{example}

The following theorem was proven in \cite{IM} and it is the main result regarding the structure of the closed formula for the Gram determinant of type $(Mb)_1$.

\begin{theorem} \label{MainTheorem}\cite{IM} 
$$\prod\limits_{k=1}^n(T_k(d)+(-1)^k z)^{\binom{2n}{n-k}} \mbox{ divides } D^{(\mathit{Mb})_1}_n(d, z, x, y, w).$$
\end{theorem}

We will extend Theorem \ref{MainTheorem} by presenting a relationship between the determinant of Type $(Mb)_1$ and the determinant of a block matrix constructed from the Gram matrix of Type $B$ and a sub-collection of elements from Type $(Mb)_1$.

\begin{definition}
    Consider the set $Mb_{n,1}$ containing the collection of all diagrams of crossingless connections between $2n$ marked points on the boundary of $Mb \hat{\times} \{0\}$ with exactly one curve intersecting the crosscap. Let $G_{n}^{Mb_{n,1}}$ be the Gram matrix defined on the set $Mb_{n,1}$ using the bilinear form $\langle \ , \ \rangle_{Mb}$,

    $$ G_{n}^{Mb_{n, 1}} = (\langle m_i, m_j \rangle_{Mb})_{1 \leq i,j \leq \binom{2n}{n-1}},$$
    where $m_i, m_j \in Mb_{n,1}$. Denote by $\tilde{G}_{n}^{Mb_{n, 1}}$ the Gram matrix obtained from substituting $y=0$ and $w=1$ into $G_{n}^{Mb_{n, 1}}$,
    $$\tilde{G}_{n}^{Mb_{n, 1}} = G_{n}^{Mb_{n, 1}}(d, z, x, y=0, w=1).$$
    Furthermore, define $\tilde{G}_{n}^{(Mb)_1}$ to be the block matrix obtained from the direct sum of the Gram determinant of type $B$ and $\tilde{G}_{n}^{Mb_{n, 1}}$,
    $$\tilde{G}_n^{(Mb)_1} = G_n^{B} \oplus \tilde{G}_{n}^{Mb_{n, 1}}.$$
\end{definition}

\begin{example}
    The Gram matrix $\tilde{G}_{3}^{Mb_{3,1}}$, shown in Table \ref{Grammatrixn3tilde}, is constructed from the elements illustrated in Figure \ref{fig:BasisM31}.

  \begin{figure}[ht]
\[  \begin{minipage}{.5in} \includegraphics[width=\textwidth]{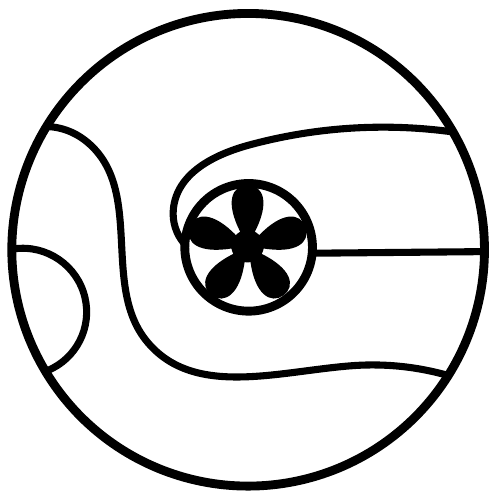} \vspace{-15pt} \[\pmb{(1 \ 2)}\] \end{minipage} 
               \qquad
        \begin{minipage}{.5in}\includegraphics[width=\textwidth]{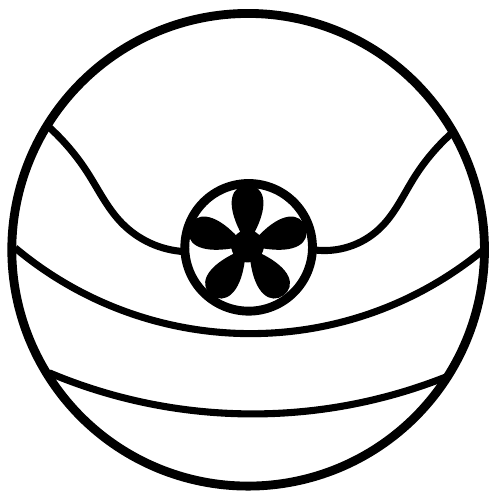} \vspace{-15pt} \[\pmb{(2 \ 3)}\] \end{minipage}
         \qquad
        \begin{minipage}{.5in}\includegraphics[width=\textwidth]{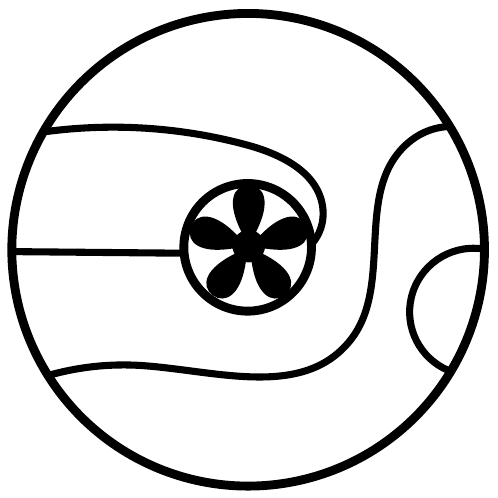} \vspace{-15pt} \[ \pmb{(3 \ 4)}\]\end{minipage}
        \qquad
        \begin{minipage}{.5in}\includegraphics[width=\textwidth]{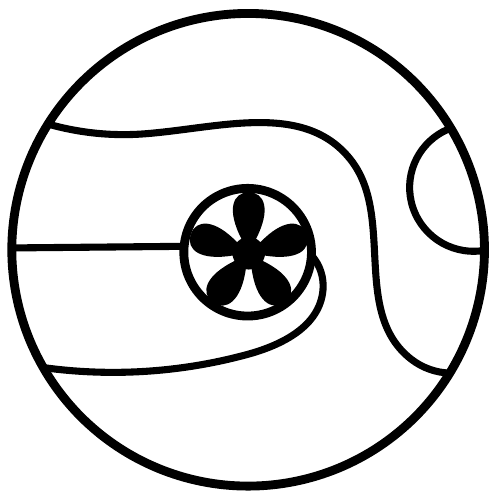} \vspace{-15pt} \[ \pmb{(4 \ 5)}\]\end{minipage}
        \qquad
        \begin{minipage}{.5in}\includegraphics[width=\textwidth]{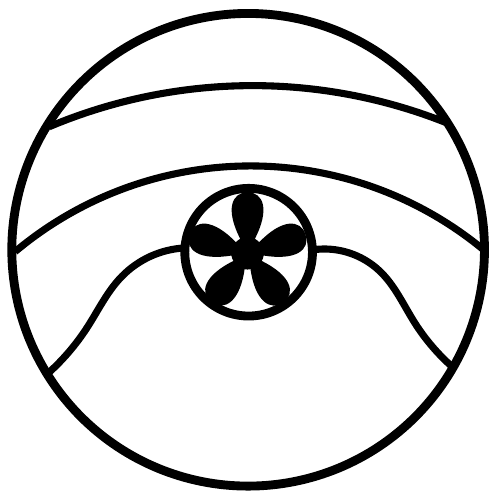} \vspace{-15pt} \[ \pmb{(5 \ 6)}\]\end{minipage}   
        \qquad
        \begin{minipage}{.5in}\includegraphics[width=\textwidth]{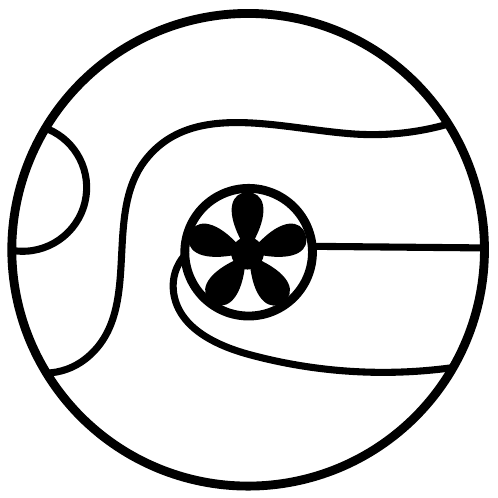} \vspace{-15pt} \[ \pmb{(1 \ 6)}\]\end{minipage} 
 \hfill \]\\
\begin{minipage}{.5in} \includegraphics[width=\textwidth]{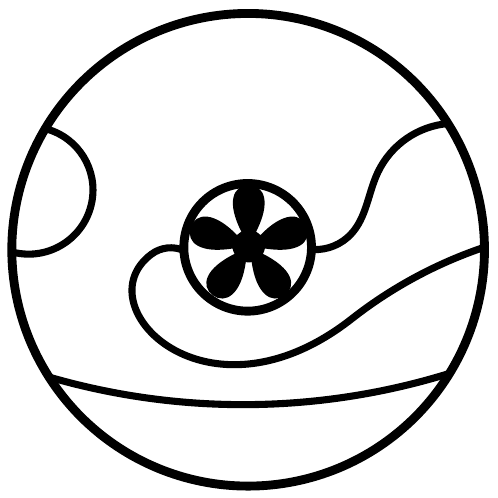} \vspace{-15pt} \[\pmb{(1 \ 3)}\] \end{minipage} 
               \qquad
        \begin{minipage}{.5in}\includegraphics[width=\textwidth]{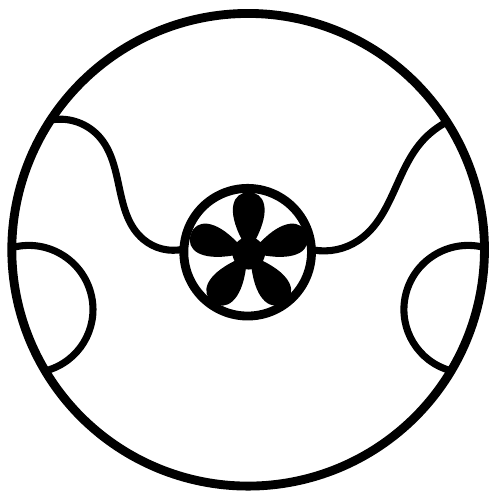} \vspace{-15pt} \[\pmb{(2 \ 4)}\] \end{minipage}
         \qquad
        \begin{minipage}{.5in}\includegraphics[width=\textwidth]{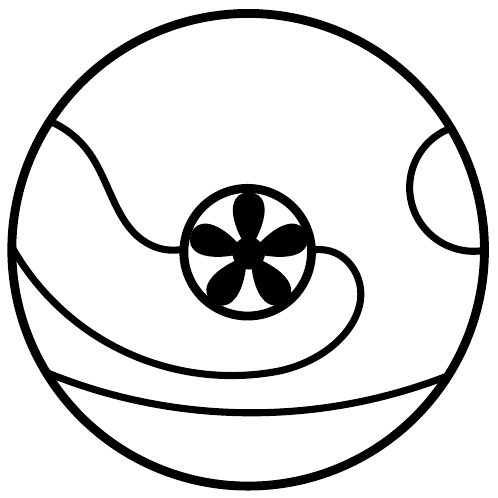} \vspace{-15pt} \[ \pmb{(3 \ 5)}\]\end{minipage}
        \qquad
        \begin{minipage}{.5in}\includegraphics[width=\textwidth]{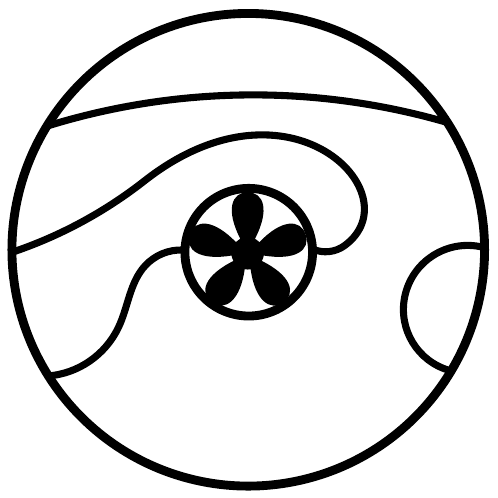} \vspace{-15pt} \[ \pmb{(4 \ 6)}\]\end{minipage}
        \qquad
        \begin{minipage}{.5in}\includegraphics[width=\textwidth]{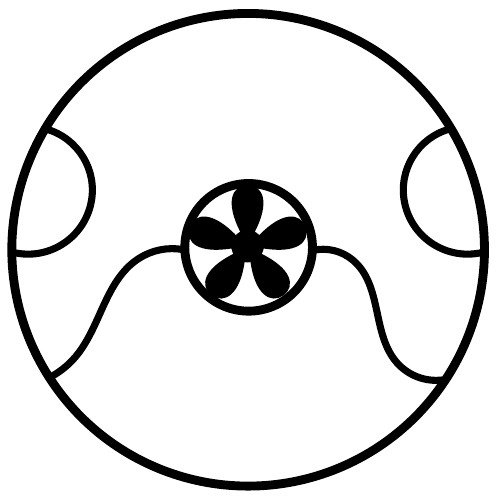} \vspace{-15pt} \[ \pmb{(1 \ 5)}\]\end{minipage}   
        \qquad
        \begin{minipage}{.5in}\includegraphics[width=\textwidth]{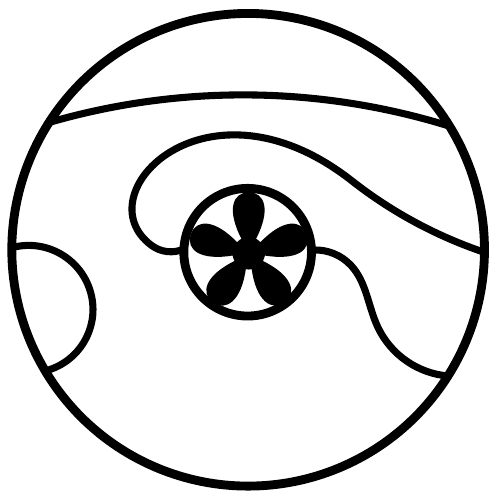} \vspace{-15pt} \[ \pmb{(2 \ 6)}\]\end{minipage} 
       \qquad
        \begin{minipage}{.5in}\includegraphics[width=\textwidth]{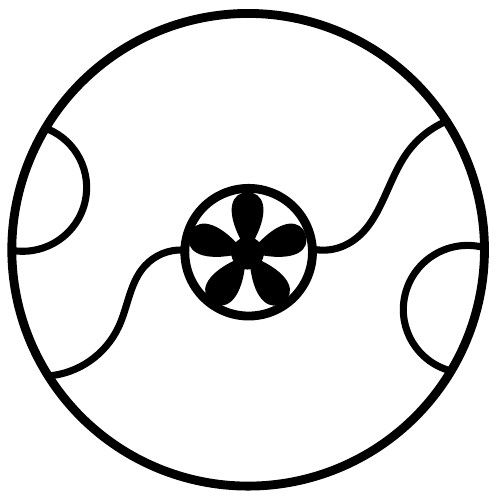} \vspace{-15pt} \[ \pmb{(1 \ 4)}\]\end{minipage}   
        \qquad
        \begin{minipage}{.5in}\includegraphics[width=\textwidth]{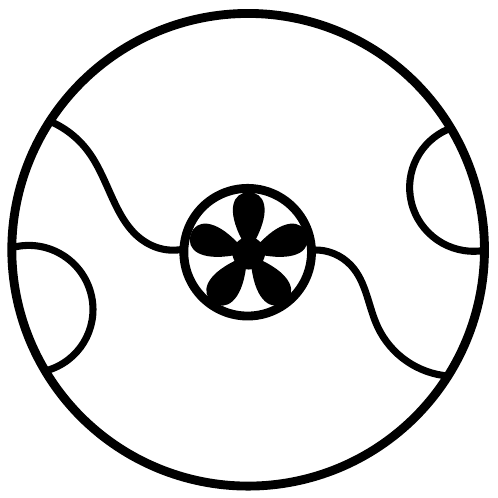} \vspace{-15pt} \[ \pmb{(2 \ 5)}\]\end{minipage}   
        \qquad
        \begin{minipage}{.5in}\includegraphics[width=\textwidth]{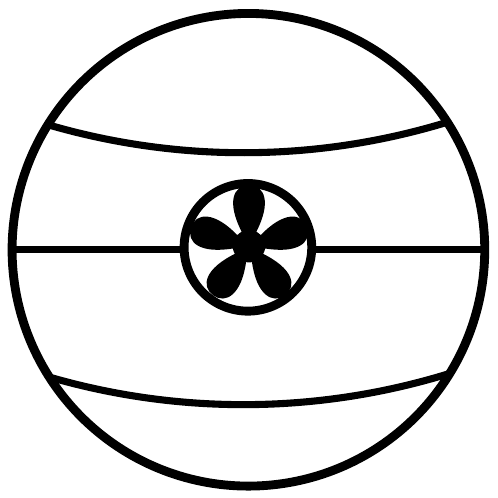} \vspace{-15pt} \[ \pmb{(3 \ 6)}\]\end{minipage}   
    \caption{The set $Mb_{3,1}$.}
    \label{fig:BasisM31}
\end{figure}

    and its determinant is

    \begin{eqnarray*}
        \det (\tilde{G}_{3}^{Mb_{3,1}}) &=& d^{12}(d-1)^2(d+1)^2(d+2)^7(d-2)^7 \\
        &=& (T_{4}(d)-2)^{\binom{6}{1}}(T_{6}(d)-2)^{\binom{6}{0}}.
    \end{eqnarray*}

    {\centering
\resizebox{\columnwidth}{!}{
\renewcommand*{\arraystretch}{1.2}\begin{tabular}{ c|||c|c|c|c|c|c||c|c|c|c|c|c||c|c|c}
 $\langle \ , \ \rangle$ & $(12)$ 
& $(23)$
& $(34)$
& $(45)$
& $(56)$
& $(16)$
& $(13)$
& $(24)$
& $(35)$
& $(46)$ 
& $(15)$
& $(26)$
& $(14)$
& $(25)$
& $(36)$\\
\hline \hline \hline 	
$(12)$ & $d^2$ & $1$ & $  1 $ & $0 $ & $1 $ & $1$ & $ d$ & $ d$ & $ 0$ & $ 0$ & $ 0$ & $ d$ & $ 1$ & $ 0$ & $ 1$ \\ \hline
 $(23)$ & $1$ & $ d^2$ & $ 1$ & $ 1$ & $ 0$ & $ 1$ & $ d$ & $ d$ & $ d$ & $ 0$ & $ 0$ & $ 0$ & $ 1$ & $ 1$ & $ 0$\\ \hline
 $(34)$ &   $1$ & $ 1$ & $ d^2$ & $ 1$ & $ 1 $ & $0 $ & $0 $ & $d $ & $d$ & $ d$ & $ 0 $ & $0 $ & $0$ & $ 1$ & $ 1$ \\ \hline
 $(45)$ &    $0 $ & $1$ & $ 1$ & $ d^2$ & $ 1$ & $ 1$ & $ 0$ & $ 0$ & $ d$ & $ d$ & $ d $ & $0 $ & $1 $ & $0 $ & $1$\\ \hline
 $(56)$ &   $1 $ & $0 $ & $1$ & $ 1$ & $ d^2$ & $ 1$ & $ 0$ & $ 0$ & $ 0$ & $ d$ & $ d$ & $ d$ & $ 1$ & $ 1$ & $ 0$ \\ \hline
 $(16)$ &   $1 $ & $1$ & $ 0$ & $ 1$ & $ 1$ & $ d^2$ & $ d$ & $ 0$ & $ 0$ & $ 0$ & $ d$ & $ d$ & $ 0$ & $ 1$ & $ 1$ \\ \hline \hline
 $(13)$   & $d $ & $d$ & $ 0$ & $ 0$ & $ 0$ & $ d$ & $ d^2$ & $ 1$ & $ 0$ & $ 1$ & $ 0$ & $ 1$ & $ d$ & $ 0$ & $ d$ \\ \hline
 $(24)$ &   $d$ & $ d$ & $ d$ & $ 0$ & $ 0$ & $ 0$ & $ 1$ & $ d^2$ & $ 1$ & $ 0$ & $ 1$ & $ 0$ & $ d$ & $ d$ & $ 0$ \\ \hline
  $(35)$ &   $  0$ & $ d$ & $ d$ & $ d$ & $ 0$ & $ 0$ & $ 0$ & $ 1$ & $ d^2$ & $ 1$ & $ 0$ & $ 1$ & $ 0$ & $ d$ & $ d$ \\ \hline
  $(46)$  &   $ 0$ & $ 0$ & $ d$ & $ d$ & $ d$ & $ 0$ & $ 1$ & $ 0$ & $ 1$ & $ d^2$ & $ 1$ & $ 0$ & $ d$ & $ 0$ & $ d$ \\ \hline
  $(15)$ &   $  0$ & $ 0$ & $ 0$ & $ d$ & $ d$ & $ d$ & $ 0$ & $ 1$ & $ 0$ & $ 1$ & $ d^2$ & $ 1$ & $ d$ & $ d$ & $ 0$\\ \hline
 $(26)$  &   $  d$ & $ 0$ & $ 0$ & $ 0$ & $ d$ & $ d$ & $ 1$ & $ 0$ & $ 1$ & $ 0$ & $ 1$ & $ d^2$ & $ 0$ & $ d$ & $ d$ \\ \hline \hline
 $(14)$  &   $  1$ & $ 1$ & $ 0$ & $ 1$ & $ 1$ & $ 0$ & $ d$ & $ d$ & $ 0$ & $ d$ & $ d$ & $ 0$ & $ d^2$ & $ 1$ & $ 1$ \\ \hline
 $(25)$ &   $   0 $ & $1 $ & $1 $ & $0 $ & $1 $ & $1 $ & $0 $ & $d$ & $ d$ & $ 0$ & $ d$ & $ d$ & $ 1$ & $ d^{2}$ & $ 1$ \\ \hline 
 $(36)$ &   $ 1 $ & $0 $ & $1 $ & $1 $ & $0 $ & $1 $ & $d $ & $0$ & $ d$ & $ d$ & $ 0$ & $ d$ & $ 1$ & $ 1$ & $d^2$
\end{tabular}}}
\captionof{table}{The Gram matrix $\tilde{G}_{3}^{Mb_{3,1}}.$}\label{Grammatrixn3tilde}

\end{example}

\begin{lemma}\label{mainlemma} Let $n \geq 2$, then 
    $$(d+z)^{\binom{2n}{n-1}}\det(G_n^{(Mb)_1}) = (w(d+z)-2xy)^{\binom{2n}{n-1}} \det(\tilde{G}_n^{(Mb)_1}). $$ 
\end{lemma}

\begin{proof}
We will adapt the proof of Proposition \ref{GRAMMB:Prop2} given in \cite{BIMP} to show that a specific column operation on $\binom{2n}{n-1}$ columns will convert $G_n^{(Mb)_1}$ into a lower triangular matrix at the expense of picking up $\binom{2n}{n-1}$ number of factors of $(d+z)$ in the determinant. \\

\begin{figure}[h]
\centering
$$
\vcenter{\hbox{
\begin{overpic}[scale = .3]{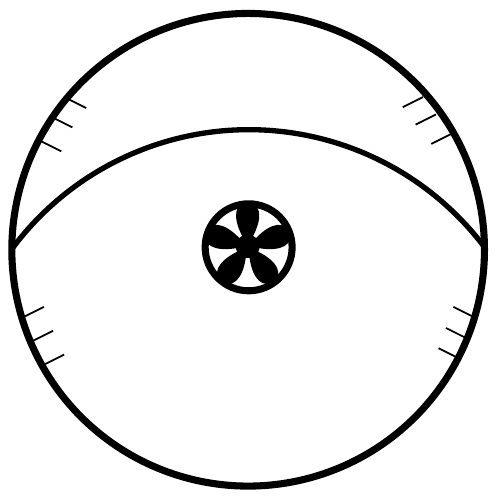}
\put(32,-7){$m_1$}
\put(32,57){$\alpha_1$}
\put(-4, 33){$i$}
\put(72, 33){$j$}
\end{overpic} }}  \ \ \ \ \ 
\vcenter{\hbox{
\begin{overpic}[scale = .3]{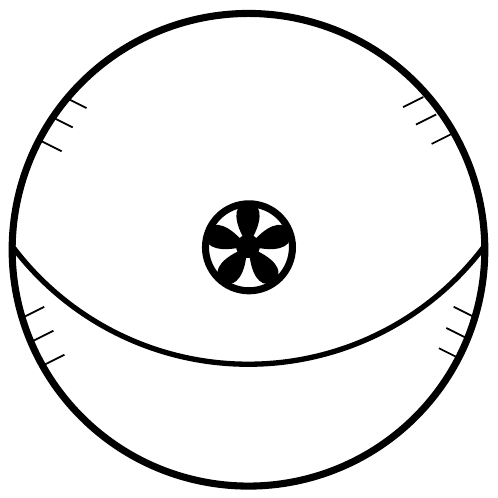}
\put(32,-7){$m_2$}
\put(32,12){$\alpha_2$}
\put(-4, 33){$i$}
\put(72, 33){$j$}
\end{overpic} }}  \ \ \ \ \
\vcenter{\hbox{
\begin{overpic}[scale = .3]{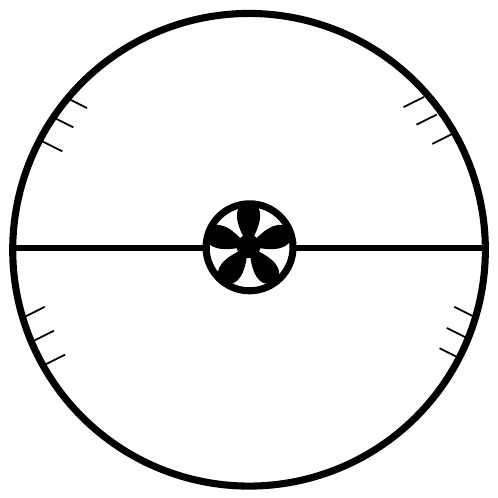}
\put(32,-7){$m$}
\put(50,38){$\alpha$}
\put(-4, 33){$i$}
\put(72, 33){$j$}
\end{overpic} }}
$$
\caption{The crossingless connection $m \in Mb_{n,1}$ with the curve $\alpha$ connected to the $i^{th}$ and $j^{th}$ marked point and passing through the crosscap. The crossingless connections $m_1,m_2 \in Mb_{n, 0}$ are obtained from $m$ by modifying $\alpha$ so that it does not pass through the crosscap.} \label{fig:triple}
\end{figure}
Let $n\geq 2$. For a crossingless connection $m$ with a curve $\alpha$ that intersects the crosscap, there are only two distinct elements that only differ by $\alpha$. They are obtained from $m$ by altering $\alpha$ in such a way that it does not intersect the crosscap; as illustrated in Figure \ref{fig:triple}. By focusing on the matrix obtained from these elements (see Equation \ref{eqn:smallMatrixblock}), we find that the entries all have a common monomial factor, say $u$, from the fact that the elements only differ by one curve. 
   
 \begin{equation}\label{eqn:smallMatrixblock}
                \begin{array}{c||c|c|c}
 \langle \ , \ \rangle_{Mb} & m_1 & m_2 & m  \\ \hline \hline
                   m_1 & ud  &  uz &  uy \\
                   m_2 &    uz  &  ud &  uy  \\
                    m &    ux  &  ux & u w      
\end{array}  \end{equation}

The row operation $m \to (d+z)m-y(m_1+m_2)$ will alter the determinant by one factor of $(d+z)$ and yield a column that is divisible by $(w(d+z)-2xy)$.

 \begin{equation}\label{eqn:smallMatrixblockrowoperation}
                \begin{array}{c||c|c|c}
 \langle \ , \ \rangle_{Mb} & m_1 & m_2 & (d+z)m-y(m_1+m_2)  \\ \hline \hline
                   m_1 & ud  &  uz &  0 \\
                   m_2 &    uz  &  ud &  0  \\
                    m &    ux  &  ux & u (w(d+z)-2xy)
\end{array}  \end{equation}

Let $c \in Mb_{n, 0}$, then the sub-row $(\langle c , m_1 \rangle_{Mb}, \langle c, m_2 \rangle_{Mb}, \langle c , m \rangle_{Mb})$ is equal to $u'(d, z, y)$ or $u'(z, d, y)$ where $u'$ is a monomial. This is because $m_1, m_2,$ and $m$ only differ by the arc $\alpha$, $c$ has no arcs intersecting the crosscap while $m_1$ and $m_2$ also have no arcs intersecting the crosscap (yielding either $d, z$ or $z, d$), and $m$ has only one arc intersecting the crosscap (yielding the variable $y$). Therefore, after applying the specified column operation on $m$ we obtain $\langle c , (d+z)m-y(m_1+m_2) \rangle_{Mb}= 0$ for all $c \in Mb_{n, 0}$. \\

Let $c' \in Mb_{n, 1}$, then the sub-row $(\langle c' , m_1 \rangle_{Mb}, \langle c', m_2 \rangle_{Mb}, \langle c' , m \rangle_{Mb})$ is equal to $u' (x, x, w)$, $u' (xd, xz, xy)$ or $u' (xz, xz, xy)$ for some monomial $u'$. Since $c$ has an arc intersecting the crosscap but $m_1$ and $m_2$ do not, then the bilinear form must pick up a curve intersecting the crosscap once; yielding the variable $x$. Since $m_1$ and $m_2$ only differ by one curve then there are two cases to consider. 

\begin{enumerate}
    \item If the $x$ variable was obtained from the arc $\alpha_1$ (or $\alpha_2$, respectively) then $\langle c' , m_1 \rangle_{Mb}= \langle c', m_2 \rangle_{Mb} = u' x$ because there is only one homotopically distinct simple closed curve that intersects the inner crosscap. Furthermore, we also have $\langle c' , m_1 \rangle_{Mb}=u'w$ where $x$ is changed to $w$ from changing $\alpha_1$ (or $\alpha_2)$ to $\alpha$. 
    \item If the $x$ variable was not obtained from the arc $\alpha_1$ (or $\alpha_2$, respectively) then the pair $(\langle c' , m_1 \rangle_{Mb}, \langle c', m_2 \rangle_{Mb})$ is equal to $u'(dx,zx)$ or $u'(zx,dx)$. Furthermore, $\langle c' , m \rangle_{Mb}=u'xy$ because $d$ (or $z$, respectively) was changed to $y$ by changing $\alpha_1$ (or $\alpha_2)$ to $\alpha$ and so this change produces a simple closed curve that intersects the outer crosscap.
\end{enumerate}
Therefore, after applying the specified column operation on $m$ if $\langle c' , m \rangle_{Mb}=u'xy$ then $\langle c' , (d+z)m-y(m_1+m_2) \rangle_{Mb}= 0$ and if $\langle c' , m \rangle_{Mb}=u_{c'}w$ then $\langle c' , (d+z)m-y(m_1+m_2) \rangle_{Mb}= u_{c'} (w(d+z)-2xy)$ for all $c' \in Mb_{n, 1}$ and some monomial $u_{c'}$ depending on $c'$.

In summary, after the column operation the determinant picks up a factor of $(d+z)$, the new column $C$ is divisible by $(w(d+z)-2xy)$, and the entries of the column obtained from elements of $Mb_{n, 0}$ are all zero. Furthermore, the column obtained from factoring $(w(d+z)-2xy)$ out of $C$ is equal to the original column after letting $w=1$ and $y=0$.

In this process we focused on a triple obtained from a crossingless connection that intersects the crosscap once. We can obtain similar triples for the $\binom{2n}{n-1}$ distinct crossingless connections in $Mb_{n, 1}$. Let $M$ be the matrix obtained by performing a column operation using the triples on the $\binom{2n}{n-1}$ columns formed from crossingless connections that intersect the crosscap once. Then $M$ is a block lower triangular matrix of the form

$$ M = \begin{pmatrix}
    G_n^B & 0 \\
    E & D
\end{pmatrix},$$
where all $\binom{2n}{n-1}$ columns of $D$ are divisible by $(w(d+z)-2xy)$ and the matrix obtained from $D$ by factoring out $(w(d+z)-2xy)$ from every column is equal to $\tilde{G}_n^{Mb_{n,1}}$. Therefore,
\begin{eqnarray*}
\det (M) &=& 
(w(d+z)-2xy)^{\binom{2n}{n-1}}\det(G_n^B)\det(\tilde{G}_n^{Mb_{n,1}})\\
&=& (w(d+z)-2xy)^{\binom{2n}{n-1}}\det(\tilde{G}_n^{(Mb)_1}).
\end{eqnarray*}
We also have $\det(M) = (d+z)^{\binom{2n}{n-1}} \det(G_n^{(Mb)_1})$ which implies that
$$(d+z)^{\binom{2n}{n-1}}\det(G_n^{(Mb)_1}) = (w(d+z)-2xy)^{\binom{2n}{n-1}} \det(\tilde{G}_n^{(Mb)_1}). $$ 

\end{proof}

\begin{theorem}\label{Maintheorem} For $n \geq 2$,

\begin{eqnarray*}
    \det(G_n^{(Mb)_1}) &=& [((w(d+z)-2xy))(d-z)]^{\binom{2n}{n-1}}\prod_{k=2}^n (T_k(d)^2-z^2)^{\binom{2n}{n-k}} \det(\tilde{G}_{n}^{Mb_{n, 1}}).
\end{eqnarray*}
\end{theorem}

\begin{proof}
Recall that, $\tilde{G}_n^{(Mb)_1} = G_n^{B} \oplus \tilde{G}_{n}^{Mb_{n, 1}}.$
    By construction and Theorem \ref{CH-P},     
    \begin{eqnarray}
    \det(\tilde{G}_n^{(Mb)_1})
&=& \det(G_n^{B}) \det( \tilde{G}_{n}^{Mb_{n, 1}}) \label{eqn:gtilde}\\    
    &=& (d+z)^{\binom{2n}{n-1}}(d-z)^{\binom{2n}{n-1}}\prod_{k=2}^n (T_k(d)^2-z^2)^{\binom{2n}{n-k}}  \nonumber \\
    && \times \det(\tilde{G}_{n}^{Mb_{n, 1}}). \nonumber
    \end{eqnarray} 
    After multiplying Equation \ref{eqn:gtilde} by $(w(d+z)-2xy)^{\binom{2n}{n-1}}$ we have
    
    \begin{eqnarray*}
  (w(d+z)-2xy)^{\binom{2n}{n-1}} \det(\tilde{G}_n^{(Mb)_1})
&=& (w(d+z)-2xy)^{\binom{2n}{n-1}} \\
&&\times (d+z)^{\binom{2n}{n-1}}(d-z)^{\binom{2n}{n-1}} \\
&& \times \prod_{k=2}^n (T_k(d)^2-z^2)^{\binom{2n}{n-k}}  \\
&&\times \det(\tilde{G}_{n}^{Mb_{n, 1}}).
    \end{eqnarray*} 
    Now, by applying Lemma \ref{mainlemma} to the left hand side of the equation we have

     \begin{eqnarray*}
 (d+z)^{\binom{2n}{n-1}}\det(G_n^{(Mb)_1})
&=& (w(d+z)-2xy)^{\binom{2n}{n-1}} \\
&& \times (d+z)^{\binom{2n}{n-1}} (d-z)^{\binom{2n}{n-1}} \\
&& \times \prod_{k=2}^n (T_k(d)^2-z^2)^{\binom{2n}{n-k}} \det(\tilde{G}_{n}^{Mb_{n, 1}}).
    \end{eqnarray*} 
    Therefore, 
     \begin{eqnarray*}
  \det(G_n^{(Mb)_1})
&=& (w(d+z)-2xy)^{\binom{2n}{n-1}} (d-z)^{\binom{2n}{n-1}} \\
&& \times \prod_{k=2}^n (T_k(d)^2-z^2)^{\binom{2n}{n-k}} \det(\tilde{G}_{n}^{Mb_{n, 1}}). \\
&=& [((w(d+z)-2xy))(d-z)]^{\binom{2n}{n-1}} \\
&& \times \prod_{k=2}^n (T_k(d)^2-z^2)^{\binom{2n}{n-k}} \det(\tilde{G}_{n}^{Mb_{n, 1}}).
    \end{eqnarray*} 
\end{proof}

The following conjecture, originally stated in \cite{IM}, provides a conjectural closed formula for the Gram determinant of type $(Mb)_1$.
\begin{conjecture}\cite{IM}\label{Conjecture}\

Let $R = \mathbb{Z}[A^{\pm 1},w,x,y,z].$ Then,
the Gram determinant of type $(Mb)_1$ for $n \geq 1$, is:
\begin{eqnarray*}
D^{(Mb)_1}_n &=&  \left[(d-z)((d + z)w -2xy) \right]^{\binom{2n}{n-1}} \prod_{k=2}^n (T_k(d)^2-z^2)^{\binom{2n}{n-k} } \prod\limits_{k=2}^n (T_{2k}(d)-2)^{\binom{2n}{n-k}},
\end{eqnarray*}

where $T_k(d)$ is the $k^{th}$ Chebyshev polynomial of the first kind and $d = -A^2-A^{-2}$.
\end{conjecture}

By applying Theorem \ref{Maintheorem}, the following conjecture implies Conjecture \ref{Conjecture}. Furthermore, notice that the product in Conjecture \ref{newconjecture} is conjectured to be a factor of $D_n^{Mb}$.

\begin{conjecture}\label{newconjecture} For $n \geq 2$,
    $$\det(\tilde{G}_{n}^{Mb_{n, 1}}) = \prod\limits_{k=2}^n (T_{2k}(d)-2)^{\binom{2n}{n-k}}.$$
\end{conjecture}

    Conjecture \ref{newconjecture} has been verified for $n=1,2,3$ and $4$. In particular the matrix $\tilde{G}_{4}^{Mb_{4, 1}}$ is provided in Table \ref{Grammatrixn4tilde}.

\section*{Acknowledgments}
The first author acknowledges the support by the Australian Research Council grant DP210103136. The second author acknowledges the support of the National Science Foundation through Grant DMS-2212736.

 {\centering
\resizebox{\columnwidth}{!}{
\renewcommand*{\arraystretch}{1.2}\begin{tabular}{ c|c|c|c|c|c|c|c||c|c|c|c|c|c|c|c||c|c|c|c|c|c|c|c||c|c|c|c|c|c|c|c||c|c|c|c|c|c|c|c||c|c|c|c|c|c|c|c||c|c|c|c|c|c|c|c}
        $d^3$ &     $1$ &     $0$ &     $1$ &     $0$ &     $1$ &     $0$ &     $1$ &     $d$ &     $1$ &     $0$ &     $1$ &     $d$ &     $1$ &     $0$ &     $1$ &     $d$ &     $1$ &     $d$ &     $0$ &     $0$ &     $1$ &     $0$ &     $d^2$ &     $0$ &     $0$ &     $d$ &     $1$ &     $d$ &     $d^2$ &     $0$ &     $1$ &     $0$ &     $1$ &     $d$ &     $d^2$ &     $d$ &     $1$ &     $0$ &     $0$ &     $d^2$ &     $d$ &     $0$ &     $0$ &     $0$ &     $0$ &     $0$ &     $d$ &     $d$ &     $0$ &     $d$ &     $1$ &     $0$ &     $0$ &     $0$ &     $1$  \\ \hline
      $1$ &     $d^3$ &     $1$ &     $0$ &     $1$ &     $0$ &     $1$ &     $0$ &     $1$ &     $d$ &     $1$ &     $0$ &     $1$ &     $d$ &     $1$ &     $0$ &     $d^2$ &     $d$ &     $1$ &     $d$ &     $0$ &     $0$ &     $1$ &     $0$ &     $1$ &     $0$ &     $0$ &     $d$ &     $1$ &     $d$ &     $d^2$ &     $0$ &     $0$ &     $0$ &     $1$ &     $d$ &     $d^2$ &     $d$ &     $1$ &     $0$ &     $d$ &     $d^2$ &     $d$ &     $0$ &     $0$ &     $0$ &     $0$ &     $0$ &     $1$ &     $d$ &     $0$ &     $d$ &     $1$ &     $0$ &     $0$ &     $0$  \\ \hline
      $0$ &     $1$ &     $d^3$ &     $1$ &     $0$ &     $1$ &     $0$ &     $1$ &     $0$ &     $1$ &     $d$ &     $1$ &     $0$ &     $1$ &     $d$ &     $1$ &     $0$ &     $d^2$ &     $d$ &     $1$ &     $d$ &     $0$ &     $0$ &     $1$ &     $0$ &     $1$ &     $0$ &     $0$ &     $d$ &     $1$ &     $d$ &     $d^2$ &     $0$ &     $0$ &     $0$ &     $1$ &     $d$ &     $d^2$ &     $d$ &     $1$ &     $0$ &     $d$ &     $d^2$ &     $d$ &     $0$ &     $0$ &     $0$ &     $0$ &     $0$ &     $1$ &     $d$ &     $0$ &     $d$ &     $1$ &     $0$ &     $0$  \\ \hline
      $1$ &     $0$ &     $1$ &     $d^3$ &     $1$ &     $0$ &     $1$ &     $0$ &     $1$ &     $0$ &     $1$ &     $d$ &     $1$ &     $0$ &     $1$ &     $d$ &     $1$ &     $0$ &     $d^2$ &     $d$ &     $1$ &     $d$ &     $0$ &     $0$ &     $d^2$ &     $0$ &     $1$ &     $0$ &     $0$ &     $d$ &     $1$ &     $d$ &     $1$ &     $0$ &     $0$ &     $0$ &     $1$ &     $d$ &     $d^2$ &     $d$ &     $0$ &     $0$ &     $d$ &     $d^2$ &     $d$ &     $0$ &     $0$ &     $0$ &     $0$ &     $0$ &     $1$ &     $d$ &     $0$ &     $d$ &     $1$ &     $0$  \\ \hline
      $0$ &     $1$ &     $0$ &     $1$ &     $d^3$ &     $1$ &     $0$ &     $1$ &     $d$ &     $1$ &     $0$ &     $1$ &     $d$ &     $1$ &     $0$ &     $1$ &     $0$ &     $1$ &     $0$ &     $d^2$ &     $d$ &     $1$ &     $d$ &     $0$ &     $d$ &     $d^2$ &     $0$ &     $1$ &     $0$ &     $0$ &     $d$ &     $1$ &     $d$ &     $1$ &     $0$ &     $0$ &     $0$ &     $1$ &     $d$ &     $d^2$ &     $0$ &     $0$ &     $0$ &     $d$ &     $d^2$ &     $d$ &     $0$ &     $0$ &     $0$ &     $0$ &     $0$ &     $1$ &     $d$ &     $0$ &     $d$ &     $1$  \\ \hline
      $1$ &     $0$ &     $1$ &     $0$ &     $1$ &     $d^3$ &     $1$ &     $0$ &     $1$ &     $d$ &     $1$ &     $0$ &     $1$ &     $d$ &     $1$ &     $0$ &     $0$ &     $0$ &     $1$ &     $0$ &     $d^2$ &     $d$ &     $1$ &     $d$ &     $1$ &     $d$ &     $d^2$ &     $0$ &     $1$ &     $0$ &     $0$ &     $d$ &     $d^2$ &     $d$ &     $1$ &     $0$ &     $0$ &     $0$ &     $1$ &     $d$ &     $0$ &     $0$ &     $0$ &     $0$ &     $d$ &     $d^2$ &     $d$ &     $0$ &     $1$ &     $0$ &     $0$ &     $0$ &     $1$ &     $d$ &     $0$ &     $d$  \\ \hline
      $0$ &     $1$ &     $0$ &     $1$ &     $0$ &     $1$ &     $d^3$ &     $1$ &     $0$ &     $1$ &     $d$ &     $1$ &     $0$ &     $1$ &     $d$ &     $1$ &     $d$ &     $0$ &     $0$ &     $1$ &     $0$ &     $d^2$ &     $d$ &     $1$ &     $d$ &     $1$ &     $d$ &     $d^2$ &     $0$ &     $1$ &     $0$ &     $0$ &     $d$ &     $d^2$ &     $d$ &     $1$ &     $0$ &     $0$ &     $0$ &     $1$ &     $0$ &     $0$ &     $0$ &     $0$ &     $0$ &     $d$ &     $d^2$ &     $d$ &     $d$ &     $1$ &     $0$ &     $0$ &     $0$ &     $1$ &     $d$ &     $0$  \\ \hline
      $1$ &     $0$ &     $1$ &     $0$ &     $1$ &     $0$ &     $1$ &     $d^3$ &     $1$ &     $0$ &     $1$ &     $d$ &     $1$ &     $0$ &     $1$ &     $d$ &     $1$ &     $d$ &     $0$ &     $0$ &     $1$ &     $0$ &     $d^2$ &     $d$ &     $0$ &     $d$ &     $1$ &     $d$ &     $d^2$ &     $0$ &     $1$ &     $0$ &     $1$ &     $d$ &     $d^2$ &     $d$ &     $1$ &     $0$ &     $0$ &     $0$ &     $d$ &     $0$ &     $0$ &     $0$ &     $0$ &     $0$ &     $d$ &     $d^2$ &     $0$ &     $d$ &     $1$ &     $0$ &     $0$ &     $0$ &     $1$ &     $d$  \\ \hline \hline 
      $d$ &     $1$ &     $0$ &     $1$ &     $d$ &     $1$ &     $0$ &     $1$ &     $d^3$ &     $1$ &     $0$ &     $1$ &     $0$ &     $1$ &     $0$ &     $1$ &     $d$ &     $d^2$ &     $0$ &     $1$ &     $d$ &     $0$ &     $0$ &     $1$ &     $d$ &     $1$ &     $0$ &     $0$ &     $d$ &     $1$ &     $0$ &     $d^2$ &     $d$ &     $0$ &     $d$ &     $1$ &     $d$ &     $0$ &     $d$ &     $1$ &     $d^2$ &     $d$ &     $0$ &     $d$ &     $d^2$ &     $0$ &     $0$ &     $0$ &     $d$ &     $1$ &     $d$ &     $d^2$ &     $d$ &     $1$ &     $d$ &     $0$  \\ \hline
      $1$ &     $d$ &     $1$ &     $0$ &     $1$ &     $d$ &     $1$ &     $0$ &     $1$ &     $d^3$ &     $1$ &     $0$ &     $1$ &     $0$ &     $1$ &     $0$ &     $1$ &     $d$ &     $d^2$ &     $0$ &     $1$ &     $d$ &     $0$ &     $0$ &     $d^2$ &     $d$ &     $1$ &     $0$ &     $0$ &     $d$ &     $1$ &     $0$ &     $1$ &     $d$ &     $0$ &     $d$ &     $1$ &     $d$ &     $0$ &     $d$ &     $0$ &     $d^2$ &     $d$ &     $0$ &     $d$ &     $d^2$ &     $0$ &     $0$ &     $0$ &     $d$ &     $1$ &     $d$ &     $d^2$ &     $d$ &     $1$ &     $d$  \\ \hline
      $0$ &     $1$ &     $d$ &     $1$ &     $0$ &     $1$ &     $d$ &     $1$ &     $0$ &     $1$ &     $d^3$ &     $1$ &     $0$ &     $1$ &     $0$ &     $1$ &     $0$ &     $1$ &     $d$ &     $d^2$ &     $0$ &     $1$ &     $d$ &     $0$ &     $0$ &     $d^2$ &     $d$ &     $1$ &     $0$ &     $0$ &     $d$ &     $1$ &     $d$ &     $1$ &     $d$ &     $0$ &     $d$ &     $1$ &     $d$ &     $0$ &     $0$ &     $0$ &     $d^2$ &     $d$ &     $0$ &     $d$ &     $d^2$ &     $0$ &     $d$ &     $0$ &     $d$ &     $1$ &     $d$ &     $d^2$ &     $d$ &     $1$  \\ \hline
      $1$ &     $0$ &     $1$ &     $d$ &     $1$ &     $0$ &     $1$ &     $d$ &     $1$ &     $0$ &     $1$ &     $d^3$ &     $1$ &     $0$ &     $1$ &     $0$ &     $0$ &     $0$ &     $1$ &     $d$ &     $d^2$ &     $0$ &     $1$ &     $d$ &     $1$ &     $0$ &     $d^2$ &     $d$ &     $1$ &     $0$ &     $0$ &     $d$ &     $0$ &     $d$ &     $1$ &     $d$ &     $0$ &     $d$ &     $1$ &     $d$ &     $0$ &     $0$ &     $0$ &     $d^2$ &     $d$ &     $0$ &     $d$ &     $d^2$ &     $1$ &     $d$ &     $0$ &     $d$ &     $1$ &     $d$ &     $d^2$ &     $d$  \\ \hline
      $d$ &     $1$ &     $0$ &     $1$ &     $d$ &     $1$ &     $0$ &     $1$ &     $0$ &     $1$ &     $0$ &     $1$ &     $d^3$ &     $1$ &     $0$ &     $1$ &     $d$ &     $0$ &     $0$ &     $1$ &     $d$ &     $d^2$ &     $0$ &     $1$ &     $d$ &     $1$ &     $0$ &     $d^2$ &     $d$ &     $1$ &     $0$ &     $0$ &     $d$ &     $0$ &     $d$ &     $1$ &     $d$ &     $0$ &     $d$ &     $1$ &     $d^2$ &     $0$ &     $0$ &     $0$ &     $d^2$ &     $d$ &     $0$ &     $d$ &     $d$ &     $1$ &     $d$ &     $0$ &     $d$ &     $1$ &     $d$ &     $d^2$  \\ \hline
      $1$ &     $d$ &     $1$ &     $0$ &     $1$ &     $d$ &     $1$ &     $0$ &     $1$ &     $0$ &     $1$ &     $0$ &     $1$ &     $d^3$ &     $1$ &     $0$ &     $1$ &     $d$ &     $0$ &     $0$ &     $1$ &     $d$ &     $d^2$ &     $0$ &     $0$ &     $d$ &     $1$ &     $0$ &     $d^2$ &     $d$ &     $1$ &     $0$ &     $1$ &     $d$ &     $0$ &     $d$ &     $1$ &     $d$ &     $0$ &     $d$ &     $d$ &     $d^2$ &     $0$ &     $0$ &     $0$ &     $d^2$ &     $d$ &     $0$ &     $d^2$ &     $d$ &     $1$ &     $d$ &     $0$ &     $d$ &     $1$ &     $d$  \\ \hline
      $0$ &     $1$ &     $d$ &     $1$ &     $0$ &     $1$ &     $d$ &     $1$ &     $0$ &     $1$ &     $0$ &     $1$ &     $0$ &     $1$ &     $d^3$ &     $1$ &     $0$ &     $1$ &     $d$ &     $0$ &     $0$ &     $1$ &     $d$ &     $d^2$ &     $0$ &     $0$ &     $d$ &     $1$ &     $0$ &     $d^2$ &     $d$ &     $1$ &     $d$ &     $1$ &     $d$ &     $0$ &     $d$ &     $1$ &     $d$ &     $0$ &     $0$ &     $d$ &     $d^2$ &     $0$ &     $0$ &     $0$ &     $d^2$ &     $d$ &     $d$ &     $d^2$ &     $d$ &     $1$ &     $d$ &     $0$ &     $d$ &     $1$  \\ \hline
      $1$ &     $0$ &     $1$ &     $d$ &     $1$ &     $0$ &     $1$ &     $d$ &     $1$ &     $0$ &     $1$ &     $0$ &     $1$ &     $0$ &     $1$ &     $d^3$ &     $d^2$ &     $0$ &     $1$ &     $d$ &     $0$ &     $0$ &     $1$ &     $d$ &     $1$ &     $0$ &     $0$ &     $d$ &     $1$ &     $0$ &     $d^2$ &     $d$ &     $0$ &     $d$ &     $1$ &     $d$ &     $0$ &     $d$ &     $1$ &     $d$ &     $d$ &     $0$ &     $d$ &     $d^2$ &     $0$ &     $0$ &     $0$ &     $d^2$ &     $1$ &     $d$ &     $d^2$ &     $d$ &     $1$ &     $d$ &     $0$ &     $d$  \\ \hline \hline
      $d$ &     $d^2$ &     $0$ &     $1$ &     $0$ &     $0$ &     $d$ &     $1$ &     $d$ &     $1$ &     $0$ &     $0$ &     $d$ &     $1$ &     $0$ &     $d^2$ &     $d^3$ &     $1$ &     $0$ &     $0$ &     $0$ &     $0$ &     $0$ &     $1$ &     $d$ &     $0$ &     $0$ &     $d^2$ &     $d$ &     $1$ &     $0$ &     $1$ &     $0$ &     $1$ &     $d$ &     $d^2$ &     $d$ &     $0$ &     $0$ &     $1$ &     $d^2$ &     $d$ &     $0$ &     $d$ &     $0$ &     $0$ &     $0$ &     $d$ &     $d$ &     $0$ &     $d$ &     $d^2$ &     $0$ &     $1$ &     $0$ &     $1$  \\ \hline
      $1$ &     $d$ &     $d^2$ &     $0$ &     $1$ &     $0$ &     $0$ &     $d$ &     $d^2$ &     $d$ &     $1$ &     $0$ &     $0$ &     $d$ &     $1$ &     $0$ &     $1$ &     $d^3$ &     $1$ &     $0$ &     $0$ &     $0$ &     $0$ &     $0$ &     $1$ &     $d$ &     $0$ &     $0$ &     $d^2$ &     $d$ &     $1$ &     $0$ &     $1$ &     $0$ &     $1$ &     $d$ &     $d^2$ &     $d$ &     $0$ &     $0$ &     $d$ &     $d^2$ &     $d$ &     $0$ &     $d$ &     $0$ &     $0$ &     $0$ &     $1$ &     $d$ &     $0$ &     $d$ &     $d^2$ &     $0$ &     $1$ &     $0$  \\ \hline
      $d$ &     $1$ &     $d$ &     $d^2$ &     $0$ &     $1$ &     $0$ &     $0$ &     $0$ &     $d^2$ &     $d$ &     $1$ &     $0$ &     $0$ &     $d$ &     $1$ &     $0$ &     $1$ &     $d^3$ &     $1$ &     $0$ &     $0$ &     $0$ &     $0$ &     $0$ &     $1$ &     $d$ &     $0$ &     $0$ &     $d^2$ &     $d$ &     $1$ &     $0$ &     $1$ &     $0$ &     $1$ &     $d$ &     $d^2$ &     $d$ &     $0$ &     $0$ &     $d$ &     $d^2$ &     $d$ &     $0$ &     $d$ &     $0$ &     $0$ &     $0$ &     $1$ &     $d$ &     $0$ &     $d$ &     $d^2$ &     $0$ &     $1$  \\ \hline
      $0$ &     $d$ &     $1$ &     $d$ &     $d^2$ &     $0$ &     $1$ &     $0$ &     $1$ &     $0$ &     $d^2$ &     $d$ &     $1$ &     $0$ &     $0$ &     $d$ &     $0$ &     $0$ &     $1$ &     $d^3$ &     $1$ &     $0$ &     $0$ &     $0$ &     $1$ &     $0$ &     $1$ &     $d$ &     $0$ &     $0$ &     $d^2$ &     $d$ &     $0$ &     $0$ &     $1$ &     $0$ &     $1$ &     $d$ &     $d^2$ &     $d$ &     $0$ &     $0$ &     $d$ &     $d^2$ &     $d$ &     $0$ &     $d$ &     $0$ &     $1$ &     $0$ &     $1$ &     $d$ &     $0$ &     $d$ &     $d^2$ &     $0$  \\ \hline
      $0$ &     $0$ &     $d$ &     $1$ &     $d$ &     $d^2$ &     $0$ &     $1$ &     $d$ &     $1$ &     $0$ &     $d^2$ &     $d$ &     $1$ &     $0$ &     $0$ &     $0$ &     $0$ &     $0$ &     $1$ &     $d^3$ &     $1$ &     $0$ &     $0$ &     $d$ &     $1$ &     $0$ &     $1$ &     $d$ &     $0$ &     $0$ &     $d^2$ &     $d$ &     $0$ &     $0$ &     $1$ &     $0$ &     $1$ &     $d$ &     $d^2$ &     $0$ &     $0$ &     $0$ &     $d$ &     $d^2$ &     $d$ &     $0$ &     $d$ &     $0$ &     $1$ &     $0$ &     $1$ &     $d$ &     $0$ &     $d$ &     $d^2$  \\ \hline
      $1$ &     $0$ &     $0$ &     $d$ &     $1$ &     $d$ &     $d^2$ &     $0$ &     $0$ &     $d$ &     $1$ &     $0$ &     $d^2$ &     $d$ &     $1$ &     $0$ &     $0$ &     $0$ &     $0$ &     $0$ &     $1$ &     $d^3$ &     $1$ &     $0$ &     $d^2$ &     $d$ &     $1$ &     $0$ &     $1$ &     $d$ &     $0$ &     $0$ &     $d^2$ &     $d$ &     $0$ &     $0$ &     $1$ &     $0$ &     $1$ &     $d$ &     $d$ &     $0$ &     $0$ &     $0$ &     $d$ &     $d^2$ &     $d$ &     $0$ &     $d^2$ &     $0$ &     $1$ &     $0$ &     $1$ &     $d$ &     $0$ &     $d$  \\ \hline
      $0$ &     $1$ &     $0$ &     $0$ &     $d$ &     $1$ &     $d$ &     $d^2$ &     $0$ &     $0$ &     $d$ &     $1$ &     $0$ &     $d^2$ &     $d$ &     $1$ &     $0$ &     $0$ &     $0$ &     $0$ &     $0$ &     $1$ &     $d^3$ &     $1$ &     $0$ &     $d^2$ &     $d$ &     $1$ &     $0$ &     $1$ &     $d$ &     $0$ &     $d$ &     $d^2$ &     $d$ &     $0$ &     $0$ &     $1$ &     $0$ &     $1$ &     $0$ &     $d$ &     $0$ &     $0$ &     $0$ &     $d$ &     $d^2$ &     $d$ &     $d$ &     $d^2$ &     $0$ &     $1$ &     $0$ &     $1$ &     $d$ &     $0$  \\ \hline
      $d^2$ &     $0$ &     $1$ &     $0$ &     $0$ &     $d$ &     $1$ &     $d$ &     $1$ &     $0$ &     $0$ &     $d$ &     $1$ &     $0$ &     $d^2$ &     $d$ &     $1$ &     $0$ &     $0$ &     $0$ &     $0$ &     $0$ &     $1$ &     $d^3$ &     $0$ &     $0$ &     $d^2$ &     $d$ &     $1$ &     $0$ &     $1$ &     $d$ &     $1$ &     $d$ &     $d^2$ &     $d$ &     $0$ &     $0$ &     $1$ &     $0$ &     $d$ &     $0$ &     $d$ &     $0$ &     $0$ &     $0$ &     $d$ &     $d^2$ &     $0$ &     $d$ &     $d^2$ &     $0$ &     $1$ &     $0$ &     $1$ &     $d$  \\ \hline \hline
      $0$ &     $1$ &     $0$ &     $d^2$ &     $d$ &     $1$ &     $d$ &     $0$ &     $d$ &     $d^2$ &     $0$ &     $1$ &     $d$ &     $0$ &     $0$ &     $1$ &     $d$ &     $1$ &     $0$ &     $1$ &     $d$ &     $d^2$ &     $0$ &     $0$ &     $d^3$ &     $1$ &     $0$ &     $0$ &     $0$ &     $0$ &     $0$ &     $1$ &     $d$ &     $1$ &     $0$ &     $1$ &     $0$ &     $0$ &     $d$ &     $d^2$ &     $0$ &     $d$ &     $0$ &     $d$ &     $d^2$ &     $d$ &     $0$ &     $0$ &     $0$ &     $1$ &     $0$ &     $d^2$ &     $d$ &     $0$ &     $d$ &     $1$  \\ \hline
      $0$ &     $0$ &     $1$ &     $0$ &     $d^2$ &     $d$ &     $1$ &     $d$ &     $1$ &     $d$ &     $d^2$ &     $0$ &     $1$ &     $d$ &     $0$ &     $0$ &     $0$ &     $d$ &     $1$ &     $0$ &     $1$ &     $d$ &     $d^2$ &     $0$ &     $1$ &     $d^3$ &     $1$ &     $0$ &     $0$ &     $0$ &     $0$ &     $0$ &     $d^2$ &     $d$ &     $1$ &     $0$ &     $1$ &     $0$ &     $0$ &     $d$ &     $0$ &     $0$ &     $d$ &     $0$ &     $d$ &     $d^2$ &     $d$ &     $0$ &     $1$ &     $0$ &     $1$ &     $0$ &     $d^2$ &     $d$ &     $0$ &     $d$  \\ \hline
      $d$ &     $0$ &     $0$ &     $1$ &     $0$ &     $d^2$ &     $d$ &     $1$ &     $0$ &     $1$ &     $d$ &     $d^2$ &     $0$ &     $1$ &     $d$ &     $0$ &     $0$ &     $0$ &     $d$ &     $1$ &     $0$ &     $1$ &     $d$ &     $d^2$ &     $0$ &     $1$ &     $d^3$ &     $1$ &     $0$ &     $0$ &     $0$ &     $0$ &     $d$ &     $d^2$ &     $d$ &     $1$ &     $0$ &     $1$ &     $0$ &     $0$ &     $0$ &     $0$ &     $0$ &     $d$ &     $0$ &     $d$ &     $d^2$ &     $d$ &     $d$ &     $1$ &     $0$ &     $1$ &     $0$ &     $d^2$ &     $d$ &     $0$  \\ \hline
      $1$ &     $d$ &     $0$ &     $0$ &     $1$ &     $0$ &     $d^2$ &     $d$ &     $0$ &     $0$ &     $1$ &     $d$ &     $d^2$ &     $0$ &     $1$ &     $d$ &     $d^2$ &     $0$ &     $0$ &     $d$ &     $1$ &     $0$ &     $1$ &     $d$ &     $0$ &     $0$ &     $1$ &     $d^3$ &     $1$ &     $0$ &     $0$ &     $0$ &     $0$ &     $d$ &     $d^2$ &     $d$ &     $1$ &     $0$ &     $1$ &     $0$ &     $d$ &     $0$ &     $0$ &     $0$ &     $d$ &     $0$ &     $d$ &     $d^2$ &     $0$ &     $d$ &     $1$ &     $0$ &     $1$ &     $0$ &     $d^2$ &     $d$  \\ \hline
      $d$ &     $1$ &     $d$ &     $0$ &     $0$ &     $1$ &     $0$ &     $d^2$ &     $d$ &     $0$ &     $0$ &     $1$ &     $d$ &     $d^2$ &     $0$ &     $1$ &     $d$ &     $d^2$ &     $0$ &     $0$ &     $d$ &     $1$ &     $0$ &     $1$ &     $0$ &     $0$ &     $0$ &     $1$ &     $d^3$ &     $1$ &     $0$ &     $0$ &     $0$ &     $0$ &     $d$ &     $d^2$ &     $d$ &     $1$ &     $0$ &     $1$ &     $d^2$ &     $d$ &     $0$ &     $0$ &     $0$ &     $d$ &     $0$ &     $d$ &     $d$ &     $0$ &     $d$ &     $1$ &     $0$ &     $1$ &     $0$ &     $d^2$  \\ \hline
      $d^2$ &     $d$ &     $1$ &     $d$ &     $0$ &     $0$ &     $1$ &     $0$ &     $1$ &     $d$ &     $0$ &     $0$ &     $1$ &     $d$ &     $d^2$ &     $0$ &     $1$ &     $d$ &     $d^2$ &     $0$ &     $0$ &     $d$ &     $1$ &     $0$ &     $0$ &     $0$ &     $0$ &     $0$ &     $1$ &     $d^3$ &     $1$ &     $0$ &     $1$ &     $0$ &     $0$ &     $d$ &     $d^2$ &     $d$ &     $1$ &     $0$ &     $d$ &     $d^2$ &     $d$ &     $0$ &     $0$ &     $0$ &     $d$ &     $0$ &     $d^2$ &     $d$ &     $0$ &     $d$ &     $1$ &     $0$ &     $1$ &     $0$  \\ \hline
      $0$ &     $d^2$ &     $d$ &     $1$ &     $d$ &     $0$ &     $0$ &     $1$ &     $0$ &     $1$ &     $d$ &     $0$ &     $0$ &     $1$ &     $d$ &     $d^2$ &     $0$ &     $1$ &     $d$ &     $d^2$ &     $0$ &     $0$ &     $d$ &     $1$ &     $0$ &     $0$ &     $0$ &     $0$ &     $0$ &     $1$ &     $d^3$ &     $1$ &     $0$ &     $1$ &     $0$ &     $0$ &     $d$ &     $d^2$ &     $d$ &     $1$ &     $0$ &     $d$ &     $d^2$ &     $d$ &     $0$ &     $0$ &     $0$ &     $d$ &     $0$ &     $d^2$ &     $d$ &     $0$ &     $d$ &     $1$ &     $0$ &     $1$  \\ \hline
      $1$ &     $0$ &     $d^2$ &     $d$ &     $1$ &     $d$ &     $0$ &     $0$ &     $d^2$ &     $0$ &     $1$ &     $d$ &     $0$ &     $0$ &     $1$ &     $d$ &     $1$ &     $0$ &     $1$ &     $d$ &     $d^2$ &     $0$ &     $0$ &     $d$ &     $1$ &     $0$ &     $0$ &     $0$ &     $0$ &     $0$ &     $1$ &     $d^3$ &     $1$ &     $0$ &     $1$ &     $0$ &     $0$ &     $d$ &     $d^2$ &     $d$ &     $d$ &     $0$ &     $d$ &     $d^2$ &     $d$ &     $0$ &     $0$ &     $0$ &     $1$ &     $0$ &     $d^2$ &     $d$ &     $0$ &     $d$ &     $1$ &     $0$  \\ \hline \hline
      $0$ &     $0$ &     $0$ &     $1$ &     $d$ &     $d^2$ &     $d$ &     $1$ &     $d$ &     $1$ &     $d$ &     $0$ &     $d$ &     $1$ &     $d$ &     $0$ &     $0$ &     $1$ &     $0$ &     $0$ &     $d$ &     $d^2$ &     $d$ &     $1$ &     $d$ &     $d^2$ &     $d$ &     $0$ &     $0$ &     $1$ &     $0$ &     $1$ &     $d^3$ &     $1$ &     $d$ &     $0$ &     $0$ &     $0$ &     $d$ &     $1$ &     $0$ &     $0$ &     $0$ &     $0$ &     $d^2$ &     $d$ &     $d^2$ &     $0$ &     $d$ &     $0$ &     $0$ &     $0$ &     $d$ &     $1$ &     $0$ &     $1$  \\ \hline
      $1$ &     $0$ &     $0$ &     $0$ &     $1$ &     $d$ &     $d^2$ &     $d$ &     $0$ &     $d$ &     $1$ &     $d$ &     $0$ &     $d$ &     $1$ &     $d$ &     $1$ &     $0$ &     $1$ &     $0$ &     $0$ &     $d$ &     $d^2$ &     $d$ &     $1$ &     $d$ &     $d^2$ &     $d$ &     $0$ &     $0$ &     $1$ &     $0$ &     $1$ &     $d^3$ &     $1$ &     $d$ &     $0$ &     $0$ &     $0$ &     $d$ &     $0$ &     $0$ &     $0$ &     $0$ &     $0$ &     $d^2$ &     $d$ &     $d^2$ &     $1$ &     $d$ &     $0$ &     $0$ &     $0$ &     $d$ &     $1$ &     $0$  \\ \hline
      $d$ &     $1$ &     $0$ &     $0$ &     $0$ &     $1$ &     $d$ &     $d^2$ &     $d$ &     $0$ &     $d$ &     $1$ &     $d$ &     $0$ &     $d$ &     $1$ &     $d$ &     $1$ &     $0$ &     $1$ &     $0$ &     $0$ &     $d$ &     $d^2$ &     $0$ &     $1$ &     $d$ &     $d^2$ &     $d$ &     $0$ &     $0$ &     $1$ &     $d$ &     $1$ &     $d^3$ &     $1$ &     $d$ &     $0$ &     $0$ &     $0$ &     $d^2$ &     $0$ &     $0$ &     $0$ &     $0$ &     $0$ &     $d^2$ &     $d$ &     $0$ &     $1$ &     $d$ &     $0$ &     $0$ &     $0$ &     $d$ &     $1$  \\ \hline
      $d^2$ &     $d$ &     $1$ &     $0$ &     $0$ &     $0$ &     $1$ &     $d$ &     $1$ &     $d$ &     $0$ &     $d$ &     $1$ &     $d$ &     $0$ &     $d$ &     $d^2$ &     $d$ &     $1$ &     $0$ &     $1$ &     $0$ &     $0$ &     $d$ &     $1$ &     $0$ &     $1$ &     $d$ &     $d^2$ &     $d$ &     $0$ &     $0$ &     $0$ &     $d$ &     $1$ &     $d^3$ &     $1$ &     $d$ &     $0$ &     $0$ &     $d$ &     $d^2$ &     $0$ &     $0$ &     $0$ &     $0$ &     $0$ &     $d^2$ &     $1$ &     $0$ &     $1$ &     $d$ &     $0$ &     $0$ &     $0$ &     $d$  \\ \hline
      $d$ &     $d^2$ &     $d$ &     $1$ &     $0$ &     $0$ &     $0$ &     $1$ &     $d$ &     $1$ &     $d$ &     $0$ &     $d$ &     $1$ &     $d$ &     $0$ &     $d$ &     $d^2$ &     $d$ &     $1$ &     $0$ &     $1$ &     $0$ &     $0$ &     $0$ &     $1$ &     $0$ &     $1$ &     $d$ &     $d^2$ &     $d$ &     $0$ &     $0$ &     $0$ &     $d$ &     $1$ &     $d^3$ &     $1$ &     $d$ &     $0$ &     $d^2$ &     $d$ &     $d^2$ &     $0$ &     $0$ &     $0$ &     $0$ &     $0$ &     $d$ &     $1$ &     $0$ &     $1$ &     $d$ &     $0$ &     $0$ &     $0$  \\ \hline
      $1$ &     $d$ &     $d^2$ &     $d$ &     $1$ &     $0$ &     $0$ &     $0$ &     $0$ &     $d$ &     $1$ &     $d$ &     $0$ &     $d$ &     $1$ &     $d$ &     $0$ &     $d$ &     $d^2$ &     $d$ &     $1$ &     $0$ &     $1$ &     $0$ &     $0$ &     $0$ &     $1$ &     $0$ &     $1$ &     $d$ &     $d^2$ &     $d$ &     $0$ &     $0$ &     $0$ &     $d$ &     $1$ &     $d^3$ &     $1$ &     $d$ &     $0$ &     $d^2$ &     $d$ &     $d^2$ &     $0$ &     $0$ &     $0$ &     $0$ &     $0$ &     $d$ &     $1$ &     $0$ &     $1$ &     $d$ &     $0$ &     $0$  \\ \hline
      $0$ &     $1$ &     $d$ &     $d^2$ &     $d$ &     $1$ &     $0$ &     $0$ &     $d$ &     $0$ &     $d$ &     $1$ &     $d$ &     $0$ &     $d$ &     $1$ &     $0$ &     $0$ &     $d$ &     $d^2$ &     $d$ &     $1$ &     $0$ &     $1$ &     $d$ &     $0$ &     $0$ &     $1$ &     $0$ &     $1$ &     $d$ &     $d^2$ &     $d$ &     $0$ &     $0$ &     $0$ &     $d$ &     $1$ &     $d^3$ &     $1$ &     $0$ &     $0$ &     $d^2$ &     $d$ &     $d^2$ &     $0$ &     $0$ &     $0$ &     $0$ &     $0$ &     $d$ &     $1$ &     $0$ &     $1$ &     $d$ &     $0$  \\ \hline
      $0$ &     $0$ &     $1$ &     $d$ &     $d^2$ &     $d$ &     $1$ &     $0$ &     $1$ &     $d$ &     $0$ &     $d$ &     $1$ &     $d$ &     $0$ &     $d$ &     $1$ &     $0$ &     $0$ &     $d$ &     $d^2$ &     $d$ &     $1$ &     $0$ &     $d^2$ &     $d$ &     $0$ &     $0$ &     $1$ &     $0$ &     $1$ &     $d$ &     $1$ &     $d$ &     $0$ &     $0$ &     $0$ &     $d$ &     $1$ &     $d^3$ &     $0$ &     $0$ &     $0$ &     $d^2$ &     $d$ &     $d^2$ &     $0$ &     $0$ &     $0$ &     $0$ &     $0$ &     $d$ &     $1$ &     $0$ &     $1$ &     $d$  \\ \hline \hline
      $d^2$ &     $d$ &     $0$ &     $0$ &     $0$ &     $0$ &     $0$ &     $d$ &     $d^2$ &     $0$ &     $0$ &     $0$ &     $d^2$ &     $d$ &     $0$ &     $d$ &     $d^2$ &     $d$ &     $0$ &     $0$ &     $0$ &     $d$ &     $0$ &     $d$ &     $0$ &     $0$ &     $0$ &     $d$ &     $d^2$ &     $d$ &     $0$ &     $d$ &     $0$ &     $0$ &     $d^2$ &     $d$ &     $d^2$ &     $0$ &     $0$ &     $0$ &     $d^3$ &     $1$ &     $0$ &     $1$ &     $0$ &     $1$ &     $0$ &     $1$ &     $d^2$ &     $0$ &     $d^2$ &     $d$ &     $0$ &     $d$ &     $0$ &     $d$  \\ \hline
      $d$ &     $d^2$ &     $d$ &     $0$ &     $0$ &     $0$ &     $0$ &     $0$ &     $d$ &     $d^2$ &     $0$ &     $0$ &     $0$ &     $d^2$ &     $d$ &     $0$ &     $d$ &     $d^2$ &     $d$ &     $0$ &     $0$ &     $0$ &     $d$ &     $0$ &     $d$ &     $0$ &     $0$ &     $0$ &     $d$ &     $d^2$ &     $d$ &     $0$ &     $0$ &     $0$ &     $0$ &     $d^2$ &     $d$ &     $d^2$ &     $0$ &     $0$ &     $1$ &     $d^3$ &     $1$ &     $0$ &     $1$ &     $0$ &     $1$ &     $0$ &     $d$ &     $d^2$ &     $0$ &     $d^2$ &     $d$ &     $0$ &     $d$ &     $0$  \\ \hline
      $0$ &     $d$ &     $d^2$ &     $d$ &     $0$ &     $0$ &     $0$ &     $0$ &     $0$ &     $d$ &     $d^2$ &     $0$ &     $0$ &     $0$ &     $d^2$ &     $d$ &     $0$ &     $d$ &     $d^2$ &     $d$ &     $0$ &     $0$ &     $0$ &     $d$ &     $0$ &     $d$ &     $0$ &     $0$ &     $0$ &     $d$ &     $d^2$ &     $d$ &     $0$ &     $0$ &     $0$ &     $0$ &     $d^2$ &     $d$ &     $d^2$ &     $0$ &     $0$ &     $1$ &     $d^3$ &     $1$ &     $0$ &     $1$ &     $0$ &     $1$ &     $0$ &     $d$ &     $d^2$ &     $0$ &     $d^2$ &     $d$ &     $0$ &     $d$  \\ \hline
      $0$ &     $0$ &     $d$ &     $d^2$ &     $d$ &     $0$ &     $0$ &     $0$ &     $d$ &     $0$ &     $d$ &     $d^2$ &     $0$ &     $0$ &     $0$ &     $d^2$ &     $d$ &     $0$ &     $d$ &     $d^2$ &     $d$ &     $0$ &     $0$ &     $0$ &     $d$ &     $0$ &     $d$ &     $0$ &     $0$ &     $0$ &     $d$ &     $d^2$ &     $0$ &     $0$ &     $0$ &     $0$ &     $0$ &     $d^2$ &     $d$ &     $d^2$ &     $1$ &     $0$ &     $1$ &     $d^3$ &     $1$ &     $0$ &     $1$ &     $0$ &     $d$ &     $0$ &     $d$ &     $d^2$ &     $0$ &     $d^2$ &     $d$ &     $0$  \\ \hline
      $0$ &     $0$ &     $0$ &     $d$ &     $d^2$ &     $d$ &     $0$ &     $0$ &     $d^2$ &     $d$ &     $0$ &     $d$ &     $d^2$ &     $0$ &     $0$ &     $0$ &     $0$ &     $d$ &     $0$ &     $d$ &     $d^2$ &     $d$ &     $0$ &     $0$ &     $d^2$ &     $d$ &     $0$ &     $d$ &     $0$ &     $0$ &     $0$ &     $d$ &     $d^2$ &     $0$ &     $0$ &     $0$ &     $0$ &     $0$ &     $d^2$ &     $d$ &     $0$ &     $1$ &     $0$ &     $1$ &     $d^3$ &     $1$ &     $0$ &     $1$ &     $0$ &     $d$ &     $0$ &     $d$ &     $d^2$ &     $0$ &     $d^2$ &     $d$  \\ \hline
      $0$ &     $0$ &     $0$ &     $0$ &     $d$ &     $d^2$ &     $d$ &     $0$ &     $0$ &     $d^2$ &     $d$ &     $0$ &     $d$ &     $d^2$ &     $0$ &     $0$ &     $0$ &     $0$ &     $d$ &     $0$ &     $d$ &     $d^2$ &     $d$ &     $0$ &     $d$ &     $d^2$ &     $d$ &     $0$ &     $d$ &     $0$ &     $0$ &     $0$ &     $d$ &     $d^2$ &     $0$ &     $0$ &     $0$ &     $0$ &     $0$ &     $d^2$ &     $1$ &     $0$ &     $1$ &     $0$ &     $1$ &     $d^3$ &     $1$ &     $0$ &     $d$ &     $0$ &     $d$ &     $0$ &     $d$ &     $d^2$ &     $0$ &     $d^2$  \\ \hline
      $0$ &     $0$ &     $0$ &     $0$ &     $0$ &     $d$ &     $d^2$ &     $d$ &     $0$ &     $0$ &     $d^2$ &     $d$ &     $0$ &     $d$ &     $d^2$ &     $0$ &     $0$ &     $0$ &     $0$ &     $d$ &     $0$ &     $d$ &     $d^2$ &     $d$ &     $0$ &     $d$ &     $d^2$ &     $d$ &     $0$ &     $d$ &     $0$ &     $0$ &     $d^2$ &     $d$ &     $d^2$ &     $0$ &     $0$ &     $0$ &     $0$ &     $0$ &     $0$ &     $1$ &     $0$ &     $1$ &     $0$ &     $1$ &     $d^3$ &     $1$ &     $d^2$ &     $d$ &     $0$ &     $d$ &     $0$ &     $d$ &     $d^2$ &     $0$  \\ \hline
      $d$ &     $0$ &     $0$ &     $0$ &     $0$ &     $0$ &     $d$ &     $d^2$ &     $0$ &     $0$ &     $0$ &     $d^2$ &     $d$ &     $0$ &     $d$ &     $d^2$ &     $d$ &     $0$ &     $0$ &     $0$ &     $d$ &     $0$ &     $d$ &     $d^2$ &     $0$ &     $0$ &     $d$ &     $d^2$ &     $d$ &     $0$ &     $d$ &     $0$ &     $0$ &     $d^2$ &     $d$ &     $d^2$ &     $0$ &     $0$ &     $0$ &     $0$ &     $1$ &     $0$ &     $1$ &     $0$ &     $1$ &     $0$ &     $1$ &     $d^3$ &     $0$ &     $d^2$ &     $d$ &     $0$ &     $d$ &     $0$ &     $d$ &     $d^2$  \\ \hline \hline
      $d$ &     $1$ &     $0$ &     $0$ &     $0$ &     $1$ &     $d$ &     $0$ &     $d$ &     $0$ &     $d$ &     $1$ &     $d$ &     $d^2$ &     $d$ &     $1$ &     $d$ &     $1$ &     $0$ &     $1$ &     $0$ &     $d^2$ &     $d$ &     $0$ &     $0$ &     $1$ &     $d$ &     $0$ &     $d$ &     $d^2$ &     $0$ &     $1$ &     $d$ &     $1$ &     $0$ &     $1$ &     $d$ &     $0$ &     $0$ &     $0$ &     $d^2$ &     $d$ &     $0$ &     $d$ &     $0$ &     $d$ &     $d^2$ &     $0$ &     $d^3$ &     $1$ &     $d$ &     $d^2$ &     $0$ &     $d^2$ &     $d$ &     $1$  \\ \hline
      $0$ &     $d$ &     $1$ &     $0$ &     $0$ &     $0$ &     $1$ &     $d$ &     $1$ &     $d$ &     $0$ &     $d$ &     $1$ &     $d$ &     $d^2$ &     $d$ &     $0$ &     $d$ &     $1$ &     $0$ &     $1$ &     $0$ &     $d^2$ &     $d$ &     $1$ &     $0$ &     $1$ &     $d$ &     $0$ &     $d$ &     $d^2$ &     $0$ &     $0$ &     $d$ &     $1$ &     $0$ &     $1$ &     $d$ &     $0$ &     $0$ &     $0$ &     $d^2$ &     $d$ &     $0$ &     $d$ &     $0$ &     $d$ &     $d^2$ &     $1$ &     $d^3$ &     $1$ &     $d$ &     $d^2$ &     $0$ &     $d^2$ &     $d$  \\ \hline
      $d$ &     $0$ &     $d$ &     $1$ &     $0$ &     $0$ &     $0$ &     $1$ &     $d$ &     $1$ &     $d$ &     $0$ &     $d$ &     $1$ &     $d$ &     $d^2$ &     $d$ &     $0$ &     $d$ &     $1$ &     $0$ &     $1$ &     $0$ &     $d^2$ &     $0$ &     $1$ &     $0$ &     $1$ &     $d$ &     $0$ &     $d$ &     $d^2$ &     $0$ &     $0$ &     $d$ &     $1$ &     $0$ &     $1$ &     $d$ &     $0$ &     $d^2$ &     $0$ &     $d^2$ &     $d$ &     $0$ &     $d$ &     $0$ &     $d$ &     $d$ &     $1$ &     $d^3$ &     $1$ &     $d$ &     $d^2$ &     $0$ &     $d^2$  \\ \hline
      $1$ &     $d$ &     $0$ &     $d$ &     $1$ &     $0$ &     $0$ &     $0$ &     $d^2$ &     $d$ &     $1$ &     $d$ &     $0$ &     $d$ &     $1$ &     $d$ &     $d^2$ &     $d$ &     $0$ &     $d$ &     $1$ &     $0$ &     $1$ &     $0$ &     $d^2$ &     $0$ &     $1$ &     $0$ &     $1$ &     $d$ &     $0$ &     $d$ &     $0$ &     $0$ &     $0$ &     $d$ &     $1$ &     $0$ &     $1$ &     $d$ &     $d$ &     $d^2$ &     $0$ &     $d^2$ &     $d$ &     $0$ &     $d$ &     $0$ &     $d^2$ &     $d$ &     $1$ &     $d^3$ &     $1$ &     $d$ &     $d^2$ &     $0$  \\ \hline
      $0$ &     $1$ &     $d$ &     $0$ &     $d$ &     $1$ &     $0$ &     $0$ &     $d$ &     $d^2$ &     $d$ &     $1$ &     $d$ &     $0$ &     $d$ &     $1$ &     $0$ &     $d^2$ &     $d$ &     $0$ &     $d$ &     $1$ &     $0$ &     $1$ &     $d$ &     $d^2$ &     $0$ &     $1$ &     $0$ &     $1$ &     $d$ &     $0$ &     $d$ &     $0$ &     $0$ &     $0$ &     $d$ &     $1$ &     $0$ &     $1$ &     $0$ &     $d$ &     $d^2$ &     $0$ &     $d^2$ &     $d$ &     $0$ &     $d$ &     $0$ &     $d^2$ &     $d$ &     $1$ &     $d^3$ &     $1$ &     $d$ &     $d^2$  \\ \hline
      $0$ &     $0$ &     $1$ &     $d$ &     $0$ &     $d$ &     $1$ &     $0$ &     $1$ &     $d$ &     $d^2$ &     $d$ &     $1$ &     $d$ &     $0$ &     $d$ &     $1$ &     $0$ &     $d^2$ &     $d$ &     $0$ &     $d$ &     $1$ &     $0$ &     $0$ &     $d$ &     $d^2$ &     $0$ &     $1$ &     $0$ &     $1$ &     $d$ &     $1$ &     $d$ &     $0$ &     $0$ &     $0$ &     $d$ &     $1$ &     $0$ &     $d$ &     $0$ &     $d$ &     $d^2$ &     $0$ &     $d^2$ &     $d$ &     $0$ &     $d^2$ &     $0$ &     $d^2$ &     $d$ &     $1$ &     $d^3$ &     $1$ &     $d$  \\ \hline
      $0$ &     $0$ &     $0$ &     $1$ &     $d$ &     $0$ &     $d$ &     $1$ &     $d$ &     $1$ &     $d$ &     $d^2$ &     $d$ &     $1$ &     $d$ &     $0$ &     $0$ &     $1$ &     $0$ &     $d^2$ &     $d$ &     $0$ &     $d$ &     $1$ &     $d$ &     $0$ &     $d$ &     $d^2$ &     $0$ &     $1$ &     $0$ &     $1$ &     $0$ &     $1$ &     $d$ &     $0$ &     $0$ &     $0$ &     $d$ &     $1$ &     $0$ &     $d$ &     $0$ &     $d$ &     $d^2$ &     $0$ &     $d^2$ &     $d$ &     $d$ &     $d^2$ &     $0$ &     $d^2$ &     $d$ &     $1$ &     $d^3$ &     $1$  \\ \hline
      $1$ &     $0$ &     $0$ &     $0$ &     $1$ &     $d$ &     $0$ &     $d$ &     $0$ &     $d$ &     $1$ &     $d$ &     $d^2$ &     $d$ &     $1$ &     $d$ &     $1$ &     $0$ &     $1$ &     $0$ &     $d^2$ &     $d$ &     $0$ &     $d$ &     $1$ &     $d$ &     $0$ &     $d$ &     $d^2$ &     $0$ &     $1$ &     $0$ &     $1$ &     $0$ &     $1$ &     $d$ &     $0$ &     $0$ &     $0$ &     $d$ &     $d$ &     $0$ &     $d$ &     $0$ &     $d$ &     $d^2$ &     $0$ &     $d^2$ &     $1$ &     $d$ &     $d^2$ &     $0$ &     $d^2$ &     $d$ &     $1$ &     $d^3$       
\end{tabular}}}
\captionof{table}{The Gram matrix $\tilde{G}_{4}^{Mb_{4,1}} .$}\label{Grammatrixn4tilde}

\bibliographystyle{abbrv}	

\bibliography{bib.bib}

@article {BIMP,
    AUTHOR = {Bakshi, Rhea Palak and Ibarra, Dionne and Mukherjee, Sujoy and
              Przytycki, J\'{o}zef H.},
     TITLE = {A generalization of the {G}ram determinant of type {A}},
   JOURNAL = {Topology Appl.},
  FJOURNAL = {Topology and its Applications},
    VOLUME = {295},
      YEAR = {2021},
     PAGES = {Paper No. 107663, 15},
      ISSN = {0166-8641},
   MRCLASS = {57K31 (05A19 57K10)},
  MRNUMBER = {4241642},
MRREVIEWER = {Qi Chen},
       DOI = {10.1016/j.topol.2021.107663},
       URL = {https://doi.org/10.1016/j.topol.2021.107663},
}

@incollection {BIMP2,
    AUTHOR = {Bakshi, Rhea Palak and Ibarra, Dionne and Mukherjee, Sujoy and
              Przytycki, J\'ozef H.},
     TITLE = {A note on the {G}ram determinant of type {M}b},
 BOOKTITLE = {Inverse problems: in memory of {P}rofessor {Z}bigniew
              {O}ziewicz},
    SERIES = {Contemp. Math.},
    VOLUME = {824},
     PAGES = {83--89},
 PUBLISHER = {Amer. Math. Soc., [Providence], RI},
      YEAR = {[2025] \copyright 2025},
      ISBN = {978-1-4704-7011-1; [9781470480448]},
   MRCLASS = {57K10 (05A19 57K31)},
  MRNUMBER = {4945859},
       DOI = {10.1090/conm/824/16451},
       URL = {https://doi.org/10.1090/conm/824/16451},
}

@article{Cai,
    author = {X. Cai} ,
    title = {A {G}ram determinant of {L}ickorish's bilinear form},
    journal =  {\it Math. Proc. Cambridge Philos. Soc.},
    year = {151 (2011), no. 1, 83–94},
  url= {\href{https://arxiv.org/pdf/1006.1297.pdf}
{arXiv:1006.1297v3} [math.GT].},
    
}

@article{CH-P,
    author = { Q. Chen and J. H. Przytycki} ,
    title = {The {G}ram determinant of the type {$B$} {T}emperley-{L}ieb algebra},
    journal =  {\it Adv. in Appl. Math.},
    year = {43(2), 2009, 156-161},
  url= {\href{https://arxiv.org/abs/0802.1083}{arXiv:0802.1083} [math.GT].},
    
}

@article {TemperleyLiebTLn,
    AUTHOR = {Temperley, H. N. V. and Lieb, E. H.},
     TITLE = {Relations between the ``percolation'' and ``colouring''
              problem and other graph-theoretical problems associated with
              regular planar lattices: some exact results for the
              ``percolation'' problem},
   JOURNAL = {Proc. Roy. Soc. London Ser. A},
  FJOURNAL = {Proceedings of the Royal Society. London. Series A.
              Mathematical, Physical and Engineering Sciences},
    VOLUME = {322},
      YEAR = {1971},
    NUMBER = {1549},
     PAGES = {251--280},
      ISSN = {0962-8444},
   MRCLASS = {05C99 (82.05)},
  MRNUMBER = {498284},
       DOI = {10.1098/rspa.1971.0067},
       URL = {https://doi.org/10.1098/rspa.1971.0067},
}

@article {KauffmanBracket,
    AUTHOR = {Kauffman, L. H.},
     TITLE = {An invariant of regular isotopy},
   JOURNAL = {Trans. Amer. Math. Soc.},
  FJOURNAL = {Transactions of the American Mathematical Society},
    VOLUME = {318},
      YEAR = {1990},
    NUMBER = {2},
     PAGES = {417--471},
      ISSN = {0002-9947},
   MRCLASS = {57M25},
  MRNUMBER = {958895},
MRREVIEWER = {Hugh Reynolds Morton},
       DOI = {10.2307/2001315},
       URL = {https://doi.org/10.2307/2001315},
}

@article {JonesIndexsubfactors,
    AUTHOR = {Jones, V. F. R.},
     TITLE = {Index for subfactors},
   JOURNAL = {Invent. Math.},
  FJOURNAL = {Inventiones Mathematicae},
    VOLUME = {72},
      YEAR = {1983},
    NUMBER = {1},
     PAGES = {1--25},
      ISSN = {0020-9910},
   MRCLASS = {46L35 (10D07)},
  MRNUMBER = {696688},
MRREVIEWER = {Yasuo Watatani},
       DOI = {10.1007/BF01389127},
       URL = {https://doi.org/10.1007/BF01389127},
}

@book {BaxterExactlysolved,
    AUTHOR = {Baxter, R. J.},
     TITLE = {Exactly solved models in statistical mechanics},
      NOTE = {Reprint of the 1982 original},
 PUBLISHER = {Academic Press, Inc. [Harcourt Brace Jovanovich, Publishers],
              London},
      YEAR = {1989},
     PAGES = {xii+486},
      ISBN = {0-12-083182-1},
   MRCLASS = {82-02 (82A05 82A69)},
  MRNUMBER = {998375},
}

@article {IM,
    AUTHOR = {D. Ibarra and G. Montoya-Vega},
     TITLE = {A new {G}ram determinant from the {M}\"obius band},
   JOURNAL = {Journal of Knot Theory and Its Ramifications},
  FJOURNAL = {Journal of Knot Theory and Its Ramifications},
volume = {33},
number = {12},
pages = {2450037},
year = {2024},
doi = {10.1142/S0218216524500378},
URL = {https://www.worldscientific.com/doi/10.1142/S0218216524500378},
}

@misc {Che,
AUTHOR = {Chen, Qi}, 
NOTE = {Personal communication (email) with J. H. Przytycki},
YEAR = {2009},
}

@incollection {DiF,
    AUTHOR = {Di Francesco, P.},
     TITLE = {The meander determinant and its generalizations},
 BOOKTITLE = {Calogero-{M}oser-{S}utherland models ({M}ontr\'{e}al, {QC}, 1997)},
    SERIES = {CRM Ser. Math. Phys.},
     PAGES = {127--144},
 PUBLISHER = {Springer, New York},
      YEAR = {2000},
   MRCLASS = {82B41 (05Axx 16S99)},
  MRNUMBER = {1843567},
MRREVIEWER = {Alexander Zvonkin},
}

@article{MS1,
    author = {P. Martin and H. Saleur},
    title = {On an algebraic approach to higher dimensional statistical mechanics},
    journal = {\it Commun. Math. Phys., 158,},
    year = {1993}
}

@article{MS2,
    author = {P. Martin and H. Saleur},
    title = {The blob algebra and the periodic {T}emperley-{L}ieb algebra},
    journal = {\it Lett. Math. Phys.},
    year = {30, 1994, no. 3, 189–206},
}

@article{Prz1,
    author = {J. H. Przytycki},
    title = {{F}undamentals of {K}auffman bracket skein modules},
    journal = {\it Kobe Math. J.},
 PAGES = {45--66},
    year = {1999},
URL= {https://arxiv.org/abs/math/9809113}
}

@book{PBIMW,
    author = {J. H. Przytycki and R. P. Bakshi and D. Ibarra and G. Montoya-Vega and D. Weeks} ,
    title = {{L}ectures in Knot Theory: An Exploration of Contemporary Topics},
    publisher = {Springer Universitext},
    year = {2024},
   PAGES = {XV, 520},
   ISBN={978-3-031-40043-8},
   DOI = {10.1007/978-3-031-40044-5},
   URL={https://link.springer.com/book/10.1007/978-3-031-40044-5#bibliographic-information},
}

@article{Wes,
    author = {B. W. Westbury} ,
    title = { The representation theory of the {T}emperley-{L}ieb algebras},
    journal = {\it Math. Z.},
    year = {219 (1995), no. 4, 539–565.}
}

\end{document}